\documentclass[pdflatex]{sn-jnl}
\usepackage{amsfonts,amsmath,amssymb}
\usepackage{dsfont}
\usepackage{xcolor}
\usepackage{amsthm}
\usepackage{siunitx}
\usepackage[T1]{fontenc}
\usepackage{lmodern}

\def\im{\mathop{\rm Im}\nolimits}
\def\diag{\mathop{\rm diag}\nolimits}
\def\re{\mathop{\rm Re}\nolimits}

\def\ker{\mathop{\rm Ker}\nolimits}

\def\rank{\mathop{\rm rank}\nolimits}
\def\nrank{\mathop{\rm nrank}\nolimits}

\newcommand{\CC}{\mathbb C}
\newcommand{\RR}{\mathbb R}

\newtheorem{theorem}{Theorem}[section]
\newtheorem{definition}[theorem]{Definition}
\newtheorem{corollary}[theorem]{Corollary}

\newtheorem{lemma}[theorem]{Lemma}
\newtheorem{remark}[theorem]{Remark}
\newtheorem*{remark*}{Remark}
\newtheorem{example}[theorem]{Example}

\usepackage[algo2e]{algorithm2e}
\usepackage{algorithmic}
\usepackage{algorithm, caption}

\begin{document}

\title{On properties and numerical computation of critical points of eigencurves of 
bivariate matrix pencils}

\author*[1,2]{Bor Plestenjak}\email{bor.plestenjak@fmf.uni-lj.si}

\affil*[1]{\orgdiv{Faculty of Mathematics and Physics}, \orgname{University of Ljubljana}, 
\orgaddress{\street{Jadranska 19}, \city{1000 Ljubljana}, \country{Slovenia}}}
\affil*[2]{\orgname{Institute of Mathematics, Physics and Mechanics}, 
\orgaddress{\street{Jadranska 19}, \city{1000 Ljubljana}, \country{Slovenia}}}

\abstract{\unboldmath We investigate critical points of eigencurves of bivariate matrix pencils
$A+\lambda B +\mu C$. 
Points $(\lambda,\mu)$ for which $\det(A+\lambda B+\mu C)=0$ form
algebraic curves in $\CC^2$ and we focus on points where $\mu'(\lambda)=0$. 
Such points are referred to as zero-group-velocity (ZGV) points, following terminology from engineering applications.
We provide a general theory for the ZGV points and show that they form a subset (with equality in the generic case) of
the 2D points $(\lambda_0,\mu_0)$, where $\lambda_0$ is a multiple eigenvalue of the pencil $(A+\mu_0 C)+\lambda B$,
or, equivalently, there exist nonzero $x$ and $y$ such that 
$(A+\lambda_0 B+\mu_0 C)x=0$, $y^H(A+\lambda_0 B+\mu_0 C)=0$, and $y^HBx=0$.

We introduce three numerical methods for 
computing 2D and ZGV points. The first method calculates all 2D (ZGV) points from the eigenvalues of 
a related singular 
two-parameter eigenvalue problem. The second method employs a projected regular two-parameter eigenvalue problem 
to compute either
all eigenvalues or only a subset of eigenvalues close to 
a given target.
The third approach is a locally convergent Gauss--Newton-type method
that computes a single 2D point from an inital approximation, the later 
can be provided for all 2D points via the method of fixed relative distance
by Jarlebring, Kvaal, and Michiels. 
In our numerical examples we use these methods to compute 2D-eigenvalues, 
solve double eigenvalue problems,
determine ZGV points of a parameter-dependent quadratic eigenvalue problem, evaluate the 
distance to instability of a stable matrix, and find
critical points of eigencurves of a two-parameter Sturm-Liouville problem.} 



\pacs[Funding]{The work was supported by the Slovenian Research and Innovation Agency (grants N1-0154 and P1-0294).}

\maketitle

\section{Introduction}
We consider parameter-dependent linear eigenvalue problems of the form 
\begin{equation}\label{eq:ABC}
    (A+\lambda B+\mu C)x=0,
\end{equation}
where $A,B,C\in\CC^{n\times n}$, $\lambda,\mu\in\CC$ and $x\in\CC^n$ is nonzero. 
We 
are interested in properties and numerical computation of points $(\lambda_0,\mu_0)\in\CC^2$ such that $\mu$ is an analytic function of $\lambda$, $\mu_0=\mu(\lambda_0)$, and $\mu'(\lambda_0)=0$. We will show that 
such points belong to a larger set of points where
$\lambda_0$ is a multiple eigenvalue of the pencil
$(A+\mu_0 C)+\lambda B$.
Such points are related to problems in applied mathematics and engineering, e.g.,  in elastodynamics \cite{ZGV_JASA_23}, 
the double eigenvalue problem \cite{Elias_DoubleEig, MP_DoubleEig}, the distance to instability of a stable matrix \cite{Freitag_Spence_LAA_2011}, and the 2D eigenvalue problem \cite{Su_Lu_Bai_2D_ArXiV}. 

To ensure that the pencil \eqref{eq:ABC} behaves regularly with respect to both parameters we assume that the problem is biregular according to the following definition. 

\begin{definition}
A bivariate matrix pencil $A+\lambda B+\mu C$ is \emph{biregular} if for each $(\lambda_0,\mu_0)\in\CC^2$
the generalized eigenvalue problems (GEPs) 
\begin{equation}\label{eq:gep_mu}
\big((A+\lambda_0 B)+\mu C\big)x=0
\end{equation}
and
\begin{equation}\label{eq:gep_lambda}
\big((A+\mu_0 C)+\lambda B\big)x=0   
\end{equation}
are both regular, i.e., $\det(A+\lambda_0 B+\mu C)\not\equiv 0$
and $\det(A+\lambda B+\mu_0 C)\not\equiv 0$.
\end{definition}

It is easy to see that a bivariate pencil $A+\lambda B+\mu C$ is biregular if and only if the characteristic polynomial 
$p(\lambda,\mu):=\det(A+\lambda B+\mu C)$ does not have a divisor of the form $\lambda-\lambda_0$ or $\mu-\mu_0$.
A sufficient condition for biregularity is that matrices $B$ and $C$ are both nonsingular.
If $A+\lambda B+\mu C$ is biregular then for each $\lambda_0\in\CC$ there are
$n$ eigenvalues $\mu\in\CC\cup\{\infty\}$ of \eqref{eq:gep_mu}. As a result,
the set of points $(\lambda,\mu)\in\CC^2$ such that $\det(A+\lambda B +\mu C)=0$ is composed of algebraic curves 
in $\CC^2$ that are called \emph{eigencurves}. If $(\lambda_0,\mu_0)\in\CC^2$ is a point on an eigencurve such that 
$\mu_0$ is a simple eigenvalue of the GEP \eqref{eq:gep_mu},  
then we can locally parameterize $\mu$ as an analytic function of $\lambda$ such that $\mu(\lambda_0)=\mu_0$ and define $\mu'(\lambda_0)$. 

\begin{theorem}\label{thm:analytic}
Let $\mu_0\in\CC$ be a simple eigenvalue of a regular GEP $((A+\lambda_0 B)+\mu C)x=0$. Then
there exist analytic functions $\mu(\lambda)$ and $x(\lambda)\ne 0$ in a neighbourhood of $\lambda_0$ such that
$\big(A+\lambda B +\mu(\lambda)C\big)x(\lambda)=0$ and 
$\mu(\lambda_0)=\mu_0$.
\end{theorem}

\begin{proof}
We know that the theorem holds for the standard eigenvalue problem where $C=I$, see, e.g.,  
\cite[Thm. 2]{Greenbaum_SIAM_Review}. Since $(A+\lambda_0 B)+\mu C$ is a regular GEP, it follows that the
matrix $M(\alpha):=\alpha(A+\lambda_0 B)+C$ is nonsingular for almost all $\alpha\in\CC$.  It is easy to see that 
$\big((A+\lambda_0 B)+\mu C\big)x=0$ if and only if 
$((A+\lambda_0 B)+\eta M(\alpha))x=0$ for $\eta=\mu/(1-\alpha \mu)$. As the theorem holds
for the standard eigenvalue problem $(M(\alpha)^{-1}(A+\lambda_0 B)+\eta I)x=0$, it thus holds for the GEP \eqref{eq:gep_mu} as well.
\end{proof}

\begin{definition}\label{def:ZGV} Let $(\lambda_0,\mu_0)\in\CC^2$ be such
that $\mu_0$ is a simple eigenvalue of a regular GEP $(A+\lambda_0 B)+\mu C)x=0$ and let 
$\mu(\lambda)$ be an analytic function in a neighbourhood of $\lambda_0$ such that
$\det(A+\lambda B +\mu(\lambda)C)=0$ and 
$\mu(\lambda_0)=\mu_0$. We say that $(\lambda_0,\mu_0)$ is a 
\emph{ZGV (zero-ground-velocity) point} of the
bivariate pencil $A+\lambda B+\mu C$ if $\mu'(\lambda_0)=0$. 
\end{definition}

The expression ZGV point comes from the study of  waves in elastodynamics,
where the angular frequency $\omega$ of a guided wave is related to the wavenumber 
$k$ via a dispersion relation $\omega(k)$. Of special interest 
are zero-group-velocity (ZGV) points $(k_*,\omega_*)$ on the dispersion curves, where 
the group velocity $c_g=\frac{\partial \omega}{\partial k}$ of a wave vanishes whereas the wave number $k_*$ remains finite, see, e.g., \cite{prada_local_2008}.

The main objective of this work is 
to clarify the algebraic structure underlying the ZGV points and 
design numerical methods that can compute all such points.
The key for this will be
the characterization
of the ZGV points as eigenvalues of an
associated singular two-parameter eigenvalue problem.

The rest of this work is structured as follows. In Section~\ref{sec:main_prop} we discuss the properties of ZGV points, introduce 2D points, and show that each ZGV point 
$(\lambda_0,\mu_0)$ is also a 2D point with $\lambda_0$ being a multiple eigenvalue of the GEP \eqref{eq:gep_lambda}. In Section~\ref{sec:aux} we review 
two-parameter eigenvalue problems (2EPs) and singular GEPs, both will be essential tools in later sections.
Section~\ref{sec:relation_2EP} contains our main result that 2D points are eigenvalues of 
a singular 2EP, which we exploit 
to derive a numerical method for computing all 2D and ZGV points. In Subsection~\ref{sec:proj_reg_2EP} we show that the 
related singular 2EP can be solved
more efficiently by projecting it into a nonsingular 2EP and applying the numerical method from \cite{HMP_SingGep2}. In Section~\ref{sec:Gauss_Newton} we provide a Gauss--Newton-type method that computes a 2D point from a good initial approximation and prove its quadratic convergence in the generic case. Approximate solutions for 2D points can be obtained by the method of fixed relative distance \cite{Elias_DoubleEig} that is presented in Section~\ref{sec:elias}. Finally, Section~\ref{sec:applications} 
explores several applications of 2D and ZGV points, illustrated with numerical examples.

\section{Main properties}\label{sec:main_prop}

The next auxiliary lemma contains well-known results (see, e.g., \cite{Gantmacher2, StewartSun_MatrixPerturbation}) on eigenvalues and eigenvectors of regular GEPs that we will use frequently in the following.

\begin{lemma}\label{lem:yBx} Let $\lambda_0\in\CC$ be a finite 
eigenvalue of a regular GEP
   $(A+\lambda B)x=0$.
\begin{enumerate}[left=0em]
\item[a)] If $\lambda_0$ is simple and 
$x_0$ and $y_0$ are the corresponding right and left eigenvector, then $y_0^HBx_0\ne 0$.
\item[b)] $\lambda_0$ is simple if and only if it is geometrically simple and 
$y_0^HBx_0\ne 0$, where 
$x_0$ and $y_0$ are the corresponding right and left eigenvector.
\item[c)] If algebraic multiplicity of $\lambda_0$ is larger than the geometric multiplicity, then there exist nonzero vectors
$x_0$ (an eigenvector) and $z_0$ (a generalized eigenvector of degree two), such that \[(A+\lambda_0 B)x_0=0\quad\textrm{and}\quad (A+\lambda_0B)z_0+Bx_0=0.\]
\end{enumerate}
\end{lemma}

The following lemma gives necessary conditions for a ZGV point.

\begin{lemma}\label{lem:2D}
If $(\lambda_0,\mu_0)\in\CC^2$ is a ZGV point of a bivariate pencil $A+\lambda B +\mu C$ and $x_0$ and $y_0$ are the corresponding right and left eigenvector,
then 
\begin{equation}\label{cond:2D}
    y_0^HBx_0=0
\end{equation}
and
\begin{equation}\label{cond_GZV_yCx}
    y_0^HCx_0\ne 0.
\end{equation}
\end{lemma} 
\begin{proof}
Let $\mu(\lambda)$ and $x(\lambda)$ be analytic functions from Theorem \ref{thm:analytic} such that
$\big(A+\lambda B +\mu(\lambda)C\big)x(\lambda)=0$, $x(\lambda_0)=x_0$, and  
$\mu(\lambda_0)=\mu_0$.
By differentiating we get 
\begin{equation}\label{eq:der}
    \big(A+\lambda B+\mu(\lambda) C\big)x'(\lambda) + \big(B +\mu'(\lambda) C\big)x(\lambda)=0.
\end{equation}
Multiplying \eqref{eq:der} from the left by $y_0^H$ and taking into account that at a ZGV point $u'(\lambda_0)=0$ yields the
condition (\ref{cond:2D}). The condition \eqref{cond_GZV_yCx} follows from Lemma \ref{lem:yBx} from the demand that $\mu_0$ is a simple
eigenvalue of the GEP \eqref{eq:gep_mu}.
\end{proof}

Based on the condition \eqref{cond:2D} from Lemma \ref{lem:2D} we introduce a larger set of 2D points for the bivariate pencil $A+\lambda B+\mu C$. The name
comes from a relation to the 2D-eigenvalue problem \cite{LuSuBai_SIMAX} that we will explore later in 
Subsection~\ref{sec:2D}.
We remark that for the symmetric matrices and a simple eigenvalue $\mu_0$ the necessary condition \eqref{cond:2D} for $(\lambda_0,\mu_0)$ to be a critical point on an eigencurve was noticed in \cite{Overton_Womersley_95}.

\begin{definition}\label{def:2D}
We say that $(\lambda_0,\mu_0)\in\CC^2$ is a \emph{2D point} of a bivariate pencil $A+\lambda B+\mu C$ if there exist nonzero vectors
$x_0$ and $y_0$ such that
\begin{equation}\label{eq:def_2D}
    \begin{split}
        (A+\lambda_0 B+\mu_0 C)x_0&=0,\\
        y_0^H(A+\lambda_0 B +\mu_0 C_0)&=0,\\
        y_0^HBx_0 & = 0.
    \end{split}
\end{equation}
\end{definition}

It follows from Lemma \ref{lem:2D} that ZGV points are a subset of 2D points. In the following we will identify which 2D points are also ZGV points.

\begin{lemma}\label{lem:dbl} Let the GEP \eqref{eq:gep_lambda}
be regular for $\mu_0\in\CC$. Then
$(\lambda_0,\mu_0)$ is a 2D point of the  bivariate pencil $A+\lambda B+\mu C$ if and only if
$\lambda_0\in\CC$ is a multiple eigenvalue of (\ref{eq:gep_lambda}).
\end{lemma}

\begin{proof}
Suppose that $\lambda_0$ is a simple eigenvalue of (\ref{eq:gep_lambda}) and $x_0$ and $y_0$ are the corresponding right and left eigenvector.  
Then  $y_0^HBx_0\ne 0$ by Lemma \ref{lem:yBx}, which contradicts the condition from Definition \ref{def:2D}.

For a proof in the other direction, let $\lambda_0$ be a multiple eigenvalue of (\ref{eq:gep_lambda}). We consider two options.
\begin{enumerate}[left=0em]
    \item[a)] $\lambda_0$ is semisimple, i.e., the geometric multiplicity of $\lambda_0$ equals the algebraic multiplicity. In this case the left and the right eigenvector belong to subspaces of dimension at least two and there exist nonzero $x_0\in\ker(A+\mu_0 C+\lambda_0 B)$ and
    $y_0\in\ker((A+\mu_0 C+\lambda_0 B)^H)$ such that $y_0^HBx_0=0$. 
    \item[b)] The algebraic multiplicity of $\lambda_0$ is larger than the geometric multiplicity. Then there exist nonzero vectors $x_0$ (an eigenvector) and $z$ (a generalized eigenvector of degree two)  such that 
\begin{equation}\label{eq:eig_root}
    \begin{split}
        (A+\lambda_0 B+\mu_0 C)x_0&=0,\\
        (A+\lambda_0 B +\mu_0 C_0)z+Bx_0&=0.
    \end{split}
\end{equation}
By multiplying the second equation of \eqref{eq:eig_root} by the left eigenvector $y_0$, we get $y_0^HBx_0=0$.    
\end{enumerate}
In both cases we get vectors $x_0$ and $y_0$ such that \eqref{eq:def_2D} holds and $(\lambda_0,\mu_0)$ is a 2D point.
\end{proof}

\begin{remark}\label{rem:abc_2D_cases} Let $(\lambda_0,\mu_0)\in\CC^2$ be a 2D point of a biregular bivariate pencil $A+\lambda B+\mu C$. If we
denote by $a_m$ and $g_m$ the algebraic and geometric multiplicity of $\lambda_0$ as an eigenvalue of \eqref{eq:gep_lambda}, respectively, then the following
options are possible:
\begin{enumerate}[left=0em]
    \item[a)] $a_m\ge 2$, $g_m=1$, and $y_0^HCx_0\ne 0$, 
    \item[b)] $a_m\ge 2$, $g_m=1$, and $y_0^HCx_0=0$, 
    \item[c)] $a_m>g_m\ge 2$,
    \item[d)] $a_m=g_m\ge 2$,
\end{enumerate}
where $x_0$ and $y_0$ are the right and the left
eigenvector for the eigenvalue $\lambda_0$ of \eqref{eq:gep_lambda}. The generic situation is a). In case d) $\lambda_0$ is a semisimple eigenvalue
of $\eqref{eq:gep_lambda}$, and a non-semisimple eigenvalue in cases a), b), c).
\end{remark}

Throughout the paper we use the
term generic as follows: a set $\mathcal A\subseteq\mathbb C^m$ is \emph{algebraic}
if it is the common zero set of finitely many complex polynomials in $m$ variables. A set 
$\mathcal U$ is called \emph{generic} if its complement is contained in a proper algebraic set. Similarly, 
a property is generic in $\mathcal U\subseteq\mathbb C^m$, if there exists
an algebraic set $\mathcal A\subseteq\mathbb C^m$ such that $\mathcal U\not\subseteq \mathcal A$
and the property holds on $\mathcal U\backslash (\mathcal A\cap \mathcal U)$.

In line with Remark~\ref{rem:abc_2D_cases} we classify 2D points into points of type a), b), c), and d). We can now show that
ZGV points of a biregular bilinear pencil 
are exactly 2D points of type a).

\begin{theorem}\label{thm:main_2D_ZGV}
Let $(\lambda_0,\mu_0)\in\CC^2$ be a 2D point of a biregular bivariate pencil $A+\lambda B+\mu C$. 
Then $(\lambda_0,\mu_0)$ is a ZGV point if and only if the
geometric multiplicity of $\lambda_0$ as an eigenvalue of \eqref{eq:gep_lambda} is one and
$y_0^HCx_0\ne 0$, where $x_0$ and $y_0$ are the right and the left
eigenvector.
\end{theorem}

\begin{proof} Let $(\lambda_0,\mu_0)$ be a ZGV point. As the geometric multiplicity of $\lambda_0$ as an eigenvalue of \eqref{eq:gep_lambda} is equal to the
geometric multiplicity of $\mu_0$ as an eigenvalue of \eqref{eq:gep_mu}, cases c) and d) from Remark~\ref{rem:abc_2D_cases} are not possible. It follows from
the condition \eqref{cond_GZV_yCx} from Lemma~\ref{lem:2D} that $(\lambda_0,\mu_0)$ is a 2D point of type a).

For a proof in the other direction, let $(\lambda_0,\mu_0)$ be a 2D point of type a). Using the same argument as above we see that the geometric
multiplicity of $\mu_0$ as an eigenvalue of \eqref{eq:gep_mu} is one. The eigenvectors $x_0$ and $y_0$ are then uniquely defined up to a multiplication by a scalar and
it follows from $y_0^HCx_0\ne 0$ that $\mu_0$ is algebraically simple. Therefore, there exist
analytic functions $\mu(\lambda)$ and $x(\lambda)$ from Theorem \ref{thm:analytic} such that
$\big(A+\lambda B +\mu(\lambda)C\big)x(\lambda)=0$, $x(\lambda_0)=x_0$, and  
$\mu(\lambda_0)=\mu_0$.
If we differentiate the equation as in the proof of Lemma~\ref{lem:2D} and multiply the derivative \eqref{eq:der} from the left by $y_0^H$, we obtain 
\[y_0^H(B+\mu'(\lambda_0)C)x_0=0.\] 
Since $(\lambda_0,\mu_0)$ is a 2D point, $y_0^HBx_0=0$, and, since $y_0^HCx_0\ne 0$, it follows that $\mu'(\lambda_0)=0$ and 
$(\lambda_0,\mu_0)$ is a ZGV point.
\end{proof}

We will use the next result 
in numerical methods to detect if a computed 2D point is a ZGV point.

\begin{corollary}\label{cor:2D_to_ZGV}
Let $(\lambda_0,\mu_0)\in\CC^2$ be a 2D point of a biregular bivariate pencil $A+\lambda B+\mu C$. 
Then $(\lambda_0,\mu_0)$ is a ZGV point if and only if 
$\mu_0$ is a simple eigenvalue of the GEP \eqref{eq:gep_mu}.
\end{corollary}
\begin{proof} 
If $\mu_0$ is a simple eigenvalue of \eqref{eq:gep_mu} and 
$x_0$ and $y_0$ are the right and the left eigenvector, then $y_0^HCx_0\ne 0$. Another thing that 
follows from $\mu_0$ being simple is that the geometric multiplicity of $\lambda_0$ as an eigenvalue of \eqref{eq:gep_lambda} is one. Thus, $(\lambda_0,\mu_0)$ is 
a 2D point of type a) and it follows from 
Theorem~\ref{thm:main_2D_ZGV} that $(\lambda_0,\mu_0)$ is a ZGV point.

A proof in the other direction  follows directly from the definition of the ZGV point, where we require that $\mu_0$ is a simple eigenvalue of \eqref{eq:gep_mu}.
\end{proof}

The following two small examples illustrate the definition of ZGV and 2D points. 

\begin{example}\rm\label{exa:one}If we take
\begin{equation}\label{ex:mat_exaone}
A=\left[\begin{matrix}1 & 2  &3 & 0\cr 2 & 0 & 1& 0\cr 3 & 1 & 1 & 0 \cr 0& 0& 0& -3\end{matrix}\right],\quad
  B=\left[\begin{matrix}1 & 0  &1 & 0\cr 0 & 1 & 1& 0\cr 1 & 1 & 0 & 0 \cr 0& 0& 0& -3\end{matrix}\right],\quad
  C=\left[\begin{matrix}2 & 1  &0 & 0\cr 1 & 3 & 0& 0\cr 0 & 0 & 1 & 0 \cr 0& 0& 0&  1\end{matrix}\right],
\end{equation}  
then the matrix $A+\lambda_0 B$ is symmetric for $\lambda_0\in\RR$ and, since $C$ is symmetric positive definite, 
all eigenvalues of $(A+\lambda_0 B)+\mu C$ are real. The left hand side of Figure \ref{fig:one} shows
the real eigencurves of $A+\lambda B +\mu C$ with the four real ZGV points
$(-2.2645, -1.3475)$, $(-1.8172,-0.17299)$, $(0.28896, 0.28248)$, and $(0.38688, 1.7975)$. 
In addition, the problem has two complex ZGV points $(-10.4081\pm 3.8258 i,7.7647\mp 2.9511 i)$. 
There are additional three 2D points of type d) that are not ZGV points. These are the intersections of eigencurves at $(-1.5330,-1.5991)$, $(-1,0)$, and $(-0.3565,1.9305)$.

\begin{figure}[h]
    \centering
    \includegraphics[width=6cm]{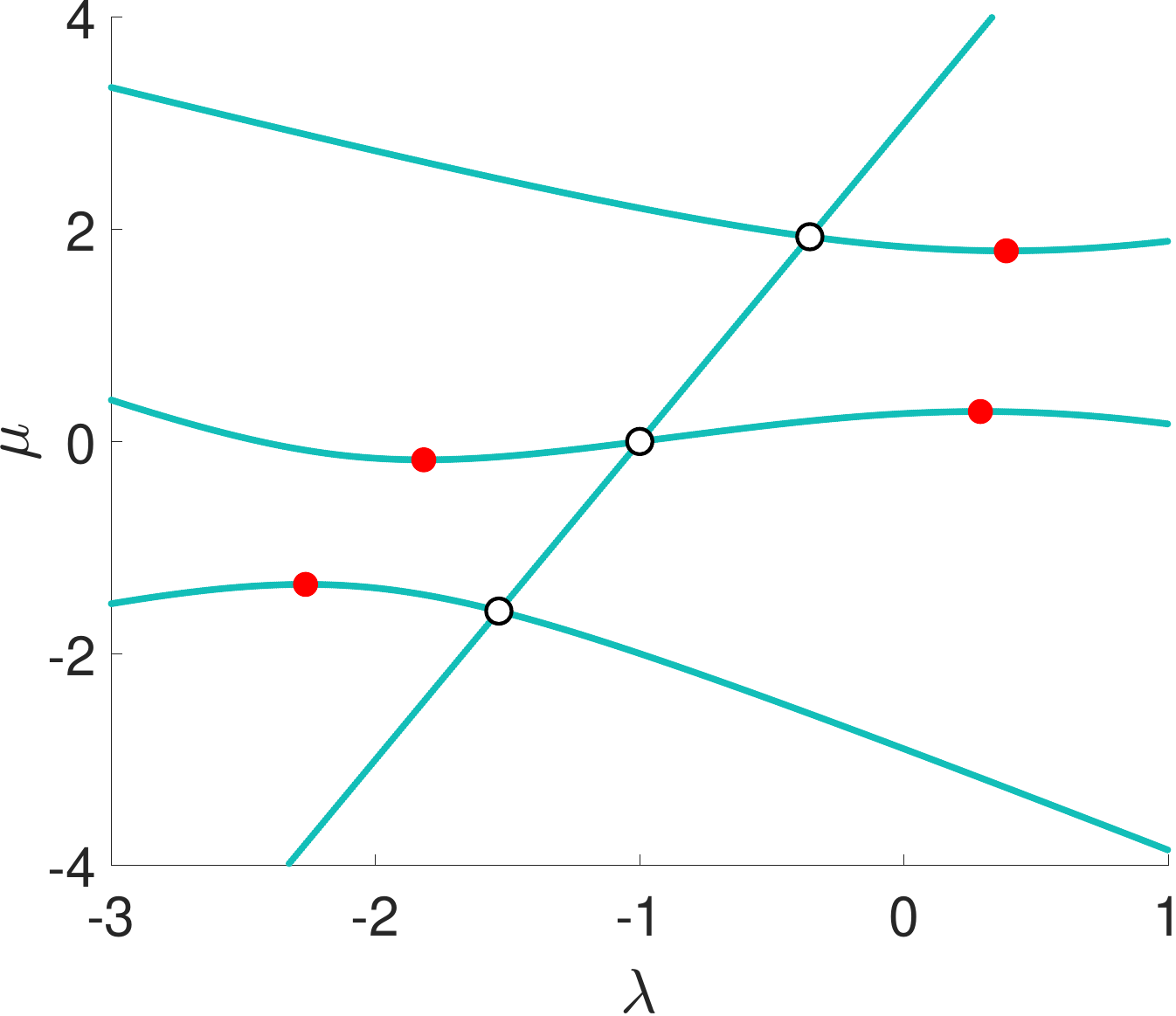}\quad
    \includegraphics[width=6cm]{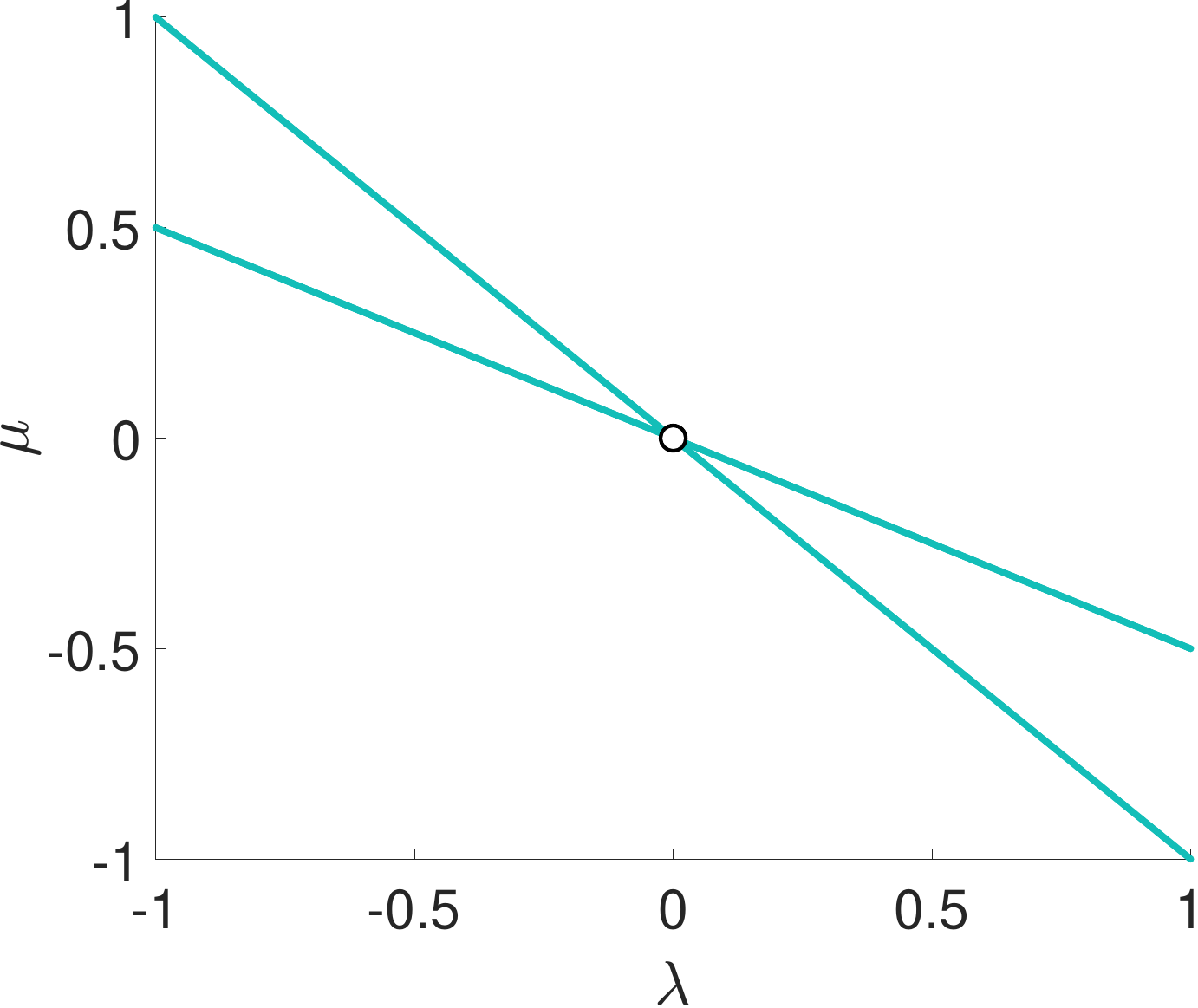}
    \caption{Real eigencurves with ZGV (red) and 2D (white) points of Examples \ref{exa:one} (left) and \ref{exa:three} (right).}
    \label{fig:one}
\end{figure}

\end{example}

\begin{example}\rm\label{exa:three}
If we take 
$$A+\lambda B+\mu C = \left[\begin{matrix}\lambda+\mu & 1  \cr 0 & \lambda + 2\mu\end{matrix}\right],$$  
then $(0,0)$ is the only 2D point. It is of type b) and the pencil does not have any ZGV points.
For the corresponding eigencurves see the right hand side of Figure \ref{fig:one}. Note that although
the 2D point is the intersection of eigencurves like in the previous example, it is of a different type than the 2D points from Example \ref{exa:one}.
\end{example}
\medskip

It follows from Corollary~\ref{cor:2D_to_ZGV} that ZGV points are exactly the subset of 2D points of 
$A+\lambda B+\mu C$ such that $\mu_0$ is a simple eigenvalue of $\big((A+\lambda_0B)+\mu C\big)x=0$. 
Based on Theorem~\ref{thm:main_2D_ZGV} and Corollary~\ref{cor:2D_to_ZGV} we can build criteria 
to identify 2D points that are also ZGV points, which means that all numerical methods for the 
computation of 2D points can be used to compute ZGV points as well. 

The problem of finding 2D points is by Lemma~\ref{lem:dbl} equivalent to 
the problem of finding 
values $\mu_0$ such that the GEP \eqref{eq:gep_lambda} has a multiple eigenvalue $\lambda_0$. Generically, see, e.g., \cite{MP_DoubleEig}, if
the GEP \eqref{eq:gep_lambda} has a multiple eigenvalue $\lambda_0$, then this eigenvalue is a derogatory double eigenvalue,
i.e., the algebraic multiplicity of $\lambda_0$ is two and the geometric multiplicity is one, thus
it is a 2D point of type a) or b). In such case, 
a 2D point is a ZGV point if and only if it satisfies the condition \eqref{cond_GZV_yCx}. As generically this condition is always satisfied,
the 2D points and ZGV points agree in the generic case, see also Theorem \ref{thm:appendix}.

We know that an $n\times n$ generic bivariate matrix 
pencil $A+\lambda B+\mu C$ has $n(n-1)$ 2D points \cite{MP_DoubleEig}, as this is the number of values $\mu$ such that a matrix $B^{-1}A+\mu B^{-1}C$ has a double eigenvalue (note that a generic $B$ is nonsingular). It follows that a biregular bivariate matrix pencil has $n(n-1)$ or less ZGV points, with an equality in the generic case when 2D points and ZGV points agree.

The numerical methods that we will present later are based on the theory on two-parameter eigenvalue problems and singular generalized eigenvalue problems that we introduce in the next section.

\section{Auxiliary results}\label{sec:aux}

A \emph{two-parameter eigenvalue problem (2EP)} has the form
\begin{equation}\label{problem}
\begin{split}
  W_1(\lambda,\mu)x_1&:=(A_1+\lambda B_1+\mu C_1)x_1=0,\\
  W_2(\lambda,\mu)x_2&:=(A_2+\lambda B_2+\mu C_2)x_2=0,
\end{split}
\end{equation}
where $A_i,B_i,C_i\in\CC^{n_i\times n_i}$, 
$x_1,x_2$ are nonzero vectors and $\lambda,\mu\in\CC$.
If $(\lambda,\mu)$ and nonzero $x_1,x_2$
satisfy
(\ref{problem}), then  $(\lambda,\mu)$ is an eigenvalue and $x_1\otimes x_2$ is the corresponding
\emph{(right) eigenvector}. Similarly, $y_1\otimes y_2$ is the
corresponding \emph{left eigenvector} if $y_i\ne 0$ and
$y_i^HW_i(\lambda,\mu)=0$ for $i=1,2$.

Note that (\ref{eq:ABC}) has a form of one equation from  
(\ref{problem}). In the generic case its solutions $(\lambda,\mu)$, 
where $\det(A+\lambda B+\mu C)=0$,
lie on curves in $\CC^2$ and there
are infinitely many such points.
On the other hand, eigenvalues of a 2EP (\ref{problem})
are the intersections of curves 
\begin{equation}\label{detproblem}
\begin{split}
  p_1(\lambda,\mu)&:=\det(W_1(\lambda,\mu))=0,\\
  p_2(\lambda,\mu)&:=\det(W_2(\lambda,\mu))=0.
\end{split}
\end{equation}
In the generic case the curves $p_1(\lambda,\mu)=0$ and 
$p_2(\lambda,\mu)=0$ do not have a nontrivial common factor and by B\'ezout's theorem (see, e.g., \cite{Shafarevich1})
it follows that (\ref{problem}) has $n_1n_2$ eigenvalues.

A 2EP (\ref{problem}) is related
to a coupled pair of GEPs
\begin{equation}\label{drugi}
\begin{split}
  \Delta_1 z&=\lambda \Delta_0 z,\\
  \Delta_2 z&=\mu \Delta_0 z,
\end{split}
\end{equation}
where $n_1n_2\times n_1n_2$ matrices
\begin{equation}\label{Deltaik1}
\begin{split}
  \Delta_0&=B_1\otimes C_2-C_1\otimes B_2,\\
  \Delta_1&=C_1\otimes A_2-A_1\otimes C_2,\\
  \Delta_2&=A_1\otimes B_2-B_1\otimes A_2
\end{split}
\end{equation}
are called \emph{operator determinants}
and $z=x_1\otimes x_2$ is a decomposable vector,
for details see, e.g., \cite{AtkinsonBook}.
A generic 2EP
(\ref{problem}) is \emph{nonsingular}, i.e., the corresponding
operator determinant $\Delta_0$ is nonsingular. In this case (see,
e.g., \cite{AtkinsonBook}), the matrices $\Delta_0^{-1}\Delta_1$ and $\Delta_0^{-1}\Delta_2$ commute and
the eigenvalues of (\ref{problem}) agree with the joint eigenvalues of (\ref{drugi}).
By exploting this relation it is possible to numerically solve a nonsingular 2EP
using standard tools for GEPs, see, e.g., \cite{HKP_JD2EP}.

If all linear combinations of $\Delta_0$, $\Delta_1$ and $\Delta_2$ are singular, 
we have a \emph{singular 2EP} that is much more difficult 
to solve. It can still have a finite number of 
 eigenvalues that can be computed by a staircase type algorithm from \cite{MP_Q2EP}
or by an algorithm for singular GEPs from \cite{HMP_SingGep,HMP_SingGep2}.
In the singular case, the relation between problems (\ref{problem}) and (\ref{drugi}) is much
less understood, for details see \cite{MP_SingMEP, KP_SingGep2}.

We say that a matrix pencil $A-\lambda B$, where $A,B\in\CC^{n\times n}$, is
\emph{singular} if $\det(A-\lambda B)\equiv 0$. This is equivalent to the
fact that the normal rank of the pencil, defined as
$$\textrm{nrank}(A-\lambda B):=\max_{\xi\in\CC}(\rank(A-\xi B)),$$
is strictly smaller than $n$. In such case $\lambda_0\in\CC$ is a finite
eigenvalue of $A-\lambda B$ if $\rank(A-\lambda_0 B)<
\textrm{nrank}(A-\lambda B)$ and $\lambda_0=\infty$ is an eigenvalue
if $\rank(B)<\textrm{nrank}(A-\lambda B)$. Standard numerical methods like,
e.g., {\tt eig} in Matlab, cannot be used to compute eigenvalues
of singular matrix pencils. The method we use is a projection to a regular problem 
of a normal rank size 
from \cite{HMP_SingGep2}.

\section{Relation to a two-parameter eigenvalue problem}\label{sec:relation_2EP}

In this section we will show the main result that 2D points are eigenvalues of
a related singular two-parameter eigenvalue problem. We will exploit this relation
to construct numerical methods that can compute all 2D points or just a subset of ZGV points of a pencil $A+\lambda B +\mu C$.

If we assume $\mu=\mu(\lambda)$ and $x=x(\lambda)$ in (\ref{eq:ABC}) and differentiate $\big(A+\lambda B+\mu(\lambda) C\big)x(\lambda)=0$, we obtain
\eqref{eq:der}.
At a ZGV point $(\lambda,\mu)$, where $\mu'(\lambda)=0$, we get
$\big(A+\lambda B+\mu(\lambda) C\big)x'(\lambda) + Bx(\lambda)=0$ that we write in a block form
\begin{equation}\label{eq:dbl_form}
 \left(\left[\begin{matrix} A & 0\cr B & A\end{matrix}\right]
 +\lambda \left[\begin{matrix} B & 0\cr 0 & B\end{matrix}\right]
 +\mu \left[\begin{matrix} C & 0\cr 0 & C\end{matrix}\right]\right)
 \left[\begin{matrix} x(\lambda) \cr x'(\lambda)\end{matrix}\right]=0.
\end{equation}
If we put together \eqref{eq:ABC} and \eqref{eq:dbl_form}, we obtain a 2EP
\begin{equation}\label{eq:2EP_ZGV}
\begin{split}
(A+\lambda B+ \mu C)x & =0\cr
\left(\left[\begin{matrix} A & 0\cr B & A\end{matrix}\right]
+\lambda \left[\begin{matrix} B & 0\cr 0 & B\end{matrix}\right]
+\mu \left[\begin{matrix} C & 0\cr 0 & C\end{matrix}\right]\right)
\left[\begin{matrix} z_1 \cr z_2\end{matrix}\right]&=0.
\end{split}
\end{equation}
Note that because of \eqref{eq:eig_root} all 2D points of type a), b), and c) give solutions of \eqref{eq:2EP_ZGV}.
The 2EP \eqref{eq:2EP_ZGV} is related to the pair of GEPs (\ref{drugi}), 
where the corresponding $2n^2\times 2n^2$ matrices
(\ref{Deltaik1}) are
\begin{equation}\label{Deltaik}
\begin{split}
  \Delta_0&=B\otimes \left[\begin{matrix} C & 0\cr 0 & C\end{matrix}\right]-C\otimes \left[\begin{matrix} B & 0\cr 0 & B\end{matrix}\right],\\
  \Delta_1&=C\otimes \left[\begin{matrix} A & 0\cr B & A\end{matrix}\right]-A\otimes \left[\begin{matrix} C & 0\cr 0 & C\end{matrix}\right],\\
  \Delta_2&=A\otimes \left[\begin{matrix} B & 0\cr 0 & B\end{matrix}\right]-B\otimes \left[\begin{matrix} A & 0\cr B & A\end{matrix}\right].
\end{split}
\end{equation}
 
The GEPs $(\Delta_1-\lambda \Delta_0)z=0$ and $(\Delta_2-\mu \Delta_0)z=0$ are singular and the 
related 2EP \eqref{eq:2EP_ZGV} is thus singular. It however has finitely many finite eigenvalues and these include not only ZGV but also 
all 2D points of (\ref{eq:ABC}). 

\begin{theorem}\label{thm:nrank_eig} Let the $n\times n$ matrices $A,B,C$ be such that $B$ and $C$ are nonsingular and for all but a 
finite number of values $\lambda_0$ the GEP \eqref{eq:gep_mu} has $n$ simple eigenvalues.
Then the following holds for the two-parameter eigenvalue problem \eqref{eq:2EP_ZGV}
and the related $2n^2\times 2n^2$ operator determinants \eqref{Deltaik}.
\begin{enumerate}[left=0em]
    \item[1)] The normal rank of the matrix pencil $\Delta_1-\lambda \Delta_0$ is $2n^2-n$.
    \item[2)] If $(\lambda_0,\mu_0)\in\CC^2$ is a 2D point for $A+\lambda B+\mu C$, then 
    $\lambda_0$ is an eigenvalue of $\Delta_1-\lambda \Delta_0$.
\end{enumerate}
\end{theorem}
\begin{proof} 
We can write 
\begin{equation}\label{eq:Delta_1_minus_Delta_0}
\Delta_1- \lambda\Delta_0 = C\otimes \left[\begin{matrix}  A+\lambda B& 0\cr B &  A+\lambda B\end{matrix}\right]
- (A+\lambda B)\otimes  \left[\begin{matrix}  C& 0\cr 0 &  C\end{matrix}\right].
\end{equation}
For a generic $\lambda_0\in\CC$ all eigenvalues of the GEP \eqref{eq:gep_mu} 
are algebraically simple. Let
$(A+\lambda_0 B+\mu_iC)q_i=0$ for $i=1,\ldots,n$, where $q_1,\ldots,q_n$ are nonzero eigenvectors and $\mu_1,\ldots,\mu_n$ are pairwise distinct eigenvalues of the 
GEP \eqref{eq:gep_mu}. Then 
it is easy to see that
\begin{equation}\label{eq:imageD0a}
(\Delta_1- \lambda_0\Delta_0)\left(q_i\otimes\left[\begin{matrix}0\cr q_j\end{matrix}\right]\right) 
= Cq_i\otimes \left[\begin{matrix}0\cr (\mu_i-\mu_j)Cq_j\end{matrix}\right]
\end{equation}
and
\begin{equation}\label{eq:imageD0b}
(\Delta_1- \lambda_0\Delta_0)\left(q_i\otimes\left[\begin{matrix}q_j\cr 0\end{matrix}\right]\right)  
= Cq_i\otimes \left[\begin{matrix}(\mu_i-\mu_j)Cq_j \cr Bq_j\end{matrix}\right]
\end{equation}
for $i,j=1,\ldots,n$. Since $B$ and $C$ are nonsingular, vectors from \eqref{eq:imageD0a} for $i\ne j$ and
vectors from \eqref{eq:imageD0b} are linearly independent, thus  
$\rank(\Delta_1-\lambda_0\Delta_0)\ge 2n^2-n$ and $\dim(\ker(\Delta_1- \lambda_0\Delta_0))\le n$.
From \eqref{eq:imageD0a} we see that
$n$ linearly independent vectors 
\begin{equation}\label{eq:veckernel}
q_i\otimes\left[\begin{matrix}0\cr q_i\end{matrix}\right], \quad i=1,\ldots,n,
\end{equation}
belong to $\ker(\Delta_1- \lambda_0\Delta_0)$, thus 
vectors \eqref{eq:veckernel} are 
basis for $\ker(\Delta_1- \lambda_0\Delta_0)$ and 
$\rank(\Delta_1-\lambda_0\Delta_0)=2n^2-n$, which proves 1). We also see that the nullspace of 
the pencil $\Delta_1-\lambda \Delta_0$, see, e.g., \cite{DeTeranDopicoMackey09}, which is the set 
\[{\cal N}=\{z(\lambda)\in\CC(\lambda)^{2n^2}:\ (\Delta_1-\lambda \Delta_0)z(\lambda)\equiv 0\},\]
contains only vectors that have $n$ zero blocks of size $n$ accordingly to \eqref{eq:veckernel}.

For 2), let $(\lambda_0,\mu_0)$ be a 2D point. We consider two options. First, if the point is of type a), b), or c), then 
option b) in the proof of Lemma~\ref{lem:dbl}
yields the existence of nonzero vectors $q$ and $z$ such that
$(A+\lambda_0B+\mu_0C)q=0$ and 
$(A+\lambda_0B+\mu_0C)z+Bq=0$. It follows that
$(A+\lambda_0B)q=-\mu_0 C q$ and $(A+\lambda_0B)z=-(Bq+\mu_0 Cz)$. It is now easy to see that 
\[
(\Delta_1- \lambda_0\Delta_0)\left(q\otimes \left[\begin{matrix}q\cr z\end{matrix}\right]\right)=0
\]
and we have a vector in the kernel of $\Delta_1-\lambda_0\Delta_0$ that is clearly linearly independent from
the vectors of the form \eqref{eq:veckernel}.
So, $\rank(\Delta_1-\lambda_0\Delta_0)<\nrank(\Delta_1,\Delta_0)$ and
$\lambda_0$ is indeed an eigenvalue of $\Delta_1-\lambda\Delta_0$. 

The remaining option is that $(\lambda_0,\mu_0)$ is a 2D point of type d), thus the geometric multiplicity $g$ of $\mu_0$ as an eigenvalue
of \eqref{eq:gep_mu} is at least $2$. If $g=a$, where $a$ is the algebraic multiplicity of $\mu_0$, and 
$q_1,\ldots,q_g$ are the corresponding linearly independent eigenvectors, then we have
additional 
$g^2-g$ vectors
\[q_i\otimes \left[\begin{matrix}0\cr q_j\end{matrix}\right]\quad {\rm for}\quad 1\le i,j\le g,\ i\ne j\]
in the kernel of $\Delta_1-\lambda_0\Delta_0$ that are linearly independent from 
vectors \eqref{eq:veckernel}. Therefore,  $\rank(\Delta_1-\lambda_0\Delta_0)<\nrank(\Delta_1,\Delta_0)$ and
$\lambda_0$ is an eigenvalue.
In case $a>g$ we could similarly construct enough linearly independent vectors 
from the kernel of $\Delta_1-\lambda_0\Delta_0$ to show that 
$\rank(\Delta_1-\lambda_0\Delta_0)<\nrank(\Delta_1,\Delta_0)$. We omit the details and refer to
\cite[Thm. 4]{MP_DoubleEig}, where a similar approach was used in the proof.
\end{proof}

The assumptions in Theorem \ref{thm:nrank_eig} are 
satisfied for generic matrices $A,B,C$, for details see Theorem \ref{thm:appendix}.
In the generic case a bivariate pencil $A+\lambda B +\mu C$ has $n(n-1)$ 2D points. It thus follows from Theorem~\ref{thm:nrank_eig} that
for generic matrices $A,B,C$ the normal rank of $\Delta_1-\lambda\Delta_0$ is $2n(n-1)$ and the related singular GEP $(\Delta_1-\lambda \Delta_0)z=0$ has $n(n-1)$ finite
eigenvalues. Each eigenvalue $\lambda_0$ of the GEP $(\Delta_1-\lambda \Delta_0)z=0$ is a candidate for the $\lambda$-coordinate of a 2D point (or a ZGV point).
To get the $\mu$-coordinate we fix $\lambda$ to $\lambda_0$ and solve the GEP \eqref{eq:gep_mu}. 
If $\mu_0$ is an eigenvalue of \eqref{eq:gep_mu} and $x$ and $y$ are the corresponding right and left eigenvector,
then:
\begin{itemize}[left=0em]
    \item[a)] $(\lambda_0,\mu_0)$ is a ZGV point if $\mu_0$ is a simple eigenvalue and 
    $y^HBx=0$. Note that if $\mu_0$ is simple then directions of vectors $x$ and $y$ are unique and the test 
    $y^HBx=0$ is well defined.
    \item[b)] $(\lambda_0,\mu_0)$ is a 2D point if the geometric multiplicity of $\mu_0$ is greater than one or if $y^HBx=0$.
\end{itemize}

Based on the above we devised Algorithm 1 that can compute all ZGV (or 2D) points. 

\noindent\vrule height 0pt depth 0.5pt width \textwidth \\
{\bf Algorithm~1: Compute ZGV (or 2D) points of pencil $A+\lambda B+\mu C$}. \\[-3mm]
\vrule height 0pt depth 0.3pt width \textwidth \\
{\bf Input:} $A,B,C\in\CC^{n\times n}$, thresholds $\delta,\eta>0$.\\
{\bf Output:} List of ZGV (or 2D) points $(\lambda_i,\mu_i)$, $i=1,\ldots,r$.\\
\begin{tabular}{ll}
{\footnotesize 1:} & Compute $2n^2\times 2n^2$ matrices $\Delta_0$, $\Delta_1$, $\Delta_2$ from (\ref{Deltaik})\\
{\footnotesize 2:} & Compute eigenvalues $\lambda_i$, $i=1,\dots,m$, of 
$\Delta_1 z = \lambda \Delta_0 z$.\\
{\footnotesize 3:} & for $i=1,\ldots,m$:\\
{\footnotesize 4:} & \quad Compute eigentriples $(\mu_j,x_j,y_j)$, $j=1,\ldots,n$, of
$(A+\lambda_i B)+\mu C$ \\
{\footnotesize 5:} & \quad for $j=1,\ldots,n$:\\
{\footnotesize 6 (ZGV):} & \quad  \quad If $|y_j^HBx_j|\le \delta\|B\|_2$,\ $\sigma_{n-1}(A+\lambda_i B+\mu_j C)>\eta
\|A+\lambda_i B+\mu_j\|_2$,\\ 
\hbox{} & \quad \quad \quad and $|y_j^HCx_j|>\delta\|C\|_2$, then add $(\lambda_i,\mu_j)$ to the list.\\
{\footnotesize 6 (2D):} & \quad \quad  If $|y_j^HBx_j|\le \delta\|B\|_2$ or 
$\sigma_{n-1}(A+\lambda_i B+\mu_j C)\le \eta \|A+\lambda_i B+\mu_j\|_2$,\\
\hbox{} & \quad \quad \quad then add $(\lambda_i,\mu_j)$ to the list.
\end{tabular} \\
\vrule height 0pt depth 0.5pt width \textwidth

In the following we give some additional details on Algorithm~1.
In line 2 we need to compute the finite eigenvalues of a singular pencil $\Delta_1-\lambda \Delta_0$. For this
purpose the projection algorithm from \cite{HMP_SingGep2} can be applied, which we will provide as Algorithm~2 in Section \ref{sec:proj_reg_2EP}. 

The test for the ZGV point in line 6 is based on Lemma \ref{lem:yBx} and reveals if $\mu_j$ is a simple eigenvalue of
$(A+\lambda_i B)+\mu C$. We detect $\mu_j$ as geometrically simple if
the second smallest singular value of $A+\lambda_i B+\mu_j C$ is large enough. 

Up to our knowledge, Algorithm 1 is the first method capable of computing all ZGV (or 2D)  points of a given pencil (\ref{eq:ABC}). As it involves a singular GEP with the $\Delta$-matrices of
size $2n^2\times 2n^2$, the method is feasible only for problems of moderate size.
In the next subsection we will show that 
for the particular problem \eqref{eq:2EP_ZGV} it is possible to construct a regular 2EP whose
eigenvalues include the eigenvalues of \eqref{eq:2EP_ZGV}. This way standard 
numerical methods for regular 2EPs can be used to compute ZGV (2D) points.

\subsection{Projected regular 2EP}\label{sec:proj_reg_2EP}

In \cite{HMP_SingGep2} a numerical method for a singular GEP using random projection of size 
of the normal rank is presented. If we apply it to compute the finite eigenvalues of the
singular pencil $\Delta_1-\lambda \Delta_0$ of normal rank $2n^2-n$ related to the 2EP \eqref{eq:2EP_ZGV}, we obtain
the following algorithm.

\noindent\vrule height 0pt depth 0.5pt width \textwidth \\
{\bf Algorithm~2: Computing finite eigenvalues of $\Delta_1-\lambda \Delta_0$
by projection}. \\[-3mm]
\vrule height 0pt depth 0.3pt width \textwidth \\
{\bf Input:} $2n^2\times 2n^2$ matrices $\Delta_0$ and $\Delta_1$ from \eqref{Deltaik},  thresholds $\delta_1,\delta_2>0$.\\
{\bf Output:} Finite eigenvalues of $\Delta_1-\lambda \Delta_0$. \\
\begin{tabular}{ll}
{\footnotesize 1:} & Select random unitary $2n^2\times 2n^2$ matrices
 $[W \ \, W_\perp]$ and 
 $[Z \ \, Z_\perp]$,
where $W$ \\ & and $Z$ have $2n^2-n$ columns.\\
{\footnotesize 2:} & Compute the eigenvalues $\lambda_i$, $i=1,\dots,2n^2-n$, and normalized right and left \\
& eigenvectors $x_i$ and $y_i$ of $W^H(\Delta_1-\lambda \Delta_0)Z$.\\
{\footnotesize 3:} & for $i=1,\ldots,2n^2-n$:\\
{\footnotesize 4:} & \quad  $\alpha_i=\|W_\perp^H(\Delta_1-\lambda_i\Delta_0)Zx_i\|_2$ and $\beta_i=\|y_i^HW^H(\Delta_1-\lambda_i\Delta_0)Z_\perp\|_2$ \\
{\footnotesize 5:} & \quad  $\gamma_i=|y_i^HW^H\Delta_0 Zx_i|\,(1+|\lambda_i|^2)^{-1/2}$\\
{\footnotesize 6:} & \quad If $\max(\alpha_i,\beta_i) < \delta_1 \, (\|\Delta_1\|_2+|\lambda_i|\,\|\Delta_0\|_2)$ 
and $\gamma_i>\delta_2$, then add $\lambda_i$ to the \\
& \quad  output list.
\end{tabular} \\
\vrule height 0pt depth 0.5pt width \textwidth
\medskip

Algorithm 2 is based on Theorem 4.5 from \cite{HMP_SingGep2} that shows that for generic matrices
$[W \ \, W_\perp]$ and $[Z \ \, Z_\perp]$ the pencil $W^H(\Delta_1-\lambda \Delta_0)Z$ is regular and 
contains all eigenvalues of the original singular pencil $\Delta_1-\lambda \Delta_0$. Moreover, 
the true finite eigenvalues of the original pencil can be separated from the true infinite eigenvalues and additional fake random eigenvalues
using criteria in lines 4, 5, and 6.

In the following we will show that it suffices to take matrices
$[W \ \, W_\perp]=I\otimes [U \ \, U_\perp]$ and 
$[Z \ \, Z_\perp] = I\otimes [V \ \, V_\perp]$,
where $[U \ \, U_\perp]$ and $[V \ \, V_\perp]$ are random unitary $2n\times 2n$ matrices such
that $U$ and $V$ have $2n-1$ columns. Due to the Kronecker structure of 
$W$ and $Z$ we can project the initial 2EP \eqref{eq:2EP_ZGV} into
a regular 2EP and thus use all available subspace methods for 2EPs, for instance
the Jacobi-Davidson method \cite{HKP_JD2EP, HP_HarmJDMEP} and the Sylvester-Arnoldi method \cite{MP_SylvArnoldi}, to compute
a small subset of eigenvalues close to a given target.
The following theorem shows that the above choice of matrices
$W$ and $Z$ leads to a regular pencil $W^H(\Delta_1-\lambda \Delta_0)Z$.

\begin{theorem}\label{thm:reg_projection}
For $n\times n$ matrices $A,B,C$ that satisfy the assumptions of Theorem \ref{thm:nrank_eig} the following holds for the two-parameter eigenvalue problem \eqref{eq:2EP_ZGV}
and the related $2n^2\times 2n^2$ operator determinants \eqref{Deltaik}. 
\begin{enumerate}[left=0em]
    \item[1)]
For generic $2n\times 2n$ unitary 
matrices $[U \ \, U_\perp]$ and $[V \ \, V_\perp]$, where
 $U$ and $V$ have $2n-1$ columns, is the two-parameter eigenvalue problem 
\begin{equation}\label{eq:proj_2EP_ZGV}
\begin{split}
(A+\lambda B+ \mu C)x_1 & =0\cr
\left(U^H\left[\begin{matrix} A & 0\cr B & A\end{matrix}\right]V
+\lambda\, U^H\left[\begin{matrix} B & 0\cr 0 & B\end{matrix}\right]V
+\mu\, U^H\left[\begin{matrix} C & 0\cr 0 & C\end{matrix}\right]V\right)
x_2&=0
\end{split}
\end{equation}
regular. 
\item[2)] If $(\lambda_0,\mu_0)\in\CC^2$ is an eigenvalue of \eqref{eq:proj_2EP_ZGV} with a right
eigenvector $x_1\otimes x_2$ and a left eigenvector $y_1\otimes y_2$, then $(\lambda_0,\mu_0)$ is 
a finite eigenvalue of \eqref{eq:2EP_ZGV} if and only if
\begin{equation}\label{projection_pog1}
    U_\perp^H\left[\begin{matrix}A+\lambda_0 B +\mu_0 C & 0 \cr B & A+\lambda_0 B +\mu_0 C\end{matrix}\right]Vx_2=0,
\end{equation}
\begin{equation}\label{projection_pog2}
 y_2^HU^H\left[\begin{matrix}A+\lambda_0 B +\mu_0 C & 0 \cr B & A+\lambda_0 B +\mu_0 C\end{matrix}\right]V_\perp=0,
\end{equation}
and
\begin{equation}\label{projection_pog3}
y_1^HBx_1\cdot y_2^HU^H\left[\begin{matrix}C & \cr & C\end{matrix}\right]Vx_2
-y_1^HCx_1\cdot y_2^HU^H\left[\begin{matrix}B & \cr & B\end{matrix}\right]Vx_2\ne 0.
\end{equation}
\end{enumerate}
\end{theorem}
\begin{proof}
For 1) it is enough to show that the matrix $(I\otimes U)^H(\Delta_1-\lambda_0\Delta_0)(I\otimes V)$ is nonsingular for a generic $\lambda_0\in\CC$. 
We know 
(see the proof of Theorem \ref{thm:nrank_eig}) 
that 
$\rank(\Delta_1-\lambda_0\Delta_0)=2n^2-n$ and  
there exist linearly independent vectors $q_1,\ldots,q_n$ 
such that 
the basis for $\ker(\Delta_1-\lambda_0\Delta_0)$ are vectors
$q_i\otimes\left[\genfrac{}{}{0pt}{}{0}{q_i}\right]$ for $i=1,\ldots,n$ (see \eqref{eq:imageD0a}) that
we assemble in a $2n^2\times n$ matrix
\[
M=\left[\begin{matrix}
q_1\otimes\left[\genfrac{}{}{0pt}{}{0}{q_1}\right] & \cdots & 
q_n\otimes\left[\genfrac{}{}{0pt}{}{0}{q_n}\right]\end{matrix}\right].
\]

Let us show that $(\Delta_1-\lambda_0\Delta_0)(I\otimes V)$ has full rank. This is 
equivalent to 
\begin{equation}\label{cnd:ker_delta_I_V}\ker(\Delta_1-\lambda_0\Delta_0)\cap\im(I\otimes V)=\{0\}.
\end{equation}
Suppose that the above is not true and 
there exists a nonzero $x\in
\im(I\otimes V)\cap \im(M)$. Since 
$\im(I\otimes V)=\im(I\otimes V_\perp)^\perp$, this is equivalent to the 
existence of a nonzero vector $r$ such that
$x=Mr$ and 
$(I\otimes V_\perp)^HMr=0$.
If $V_\perp=\left[\genfrac{}{}{0pt}{}{v_a}{v_b}\right]$ for $v_a,v_b\in\CC^n$,
we get 
\begin{equation}\label{eq:MQW}
(I\otimes V_\perp)^HM=Q\diag(v_b^Hq_1,\ldots,v_b^Hq_n),\end{equation}
where $Q=\left[\begin{matrix}q_1 & \cdots & q_n\end{matrix}\right]$. Since in the generic case 
$v_b^Hq_i\ne 0$ for $i=1,\ldots,n$, is
the matrix \eqref{eq:MQW} nonsingular and such $r$ cannot exist. Therefore, $(\Delta_1-\lambda_0\Delta_0)(I\otimes V)$ has full rank and
$(\Delta_1-\lambda_0\Delta_0)(I\otimes V)z\ne 0$ for all $z\ne 0$. 

In a similar way we can show that $\ker((I\otimes U)^H)\cap\im(\Delta_1-\lambda_0\Delta_0)=\{0\}$. This is equivalent to the 
condition $\im(I\otimes U)\cap\ker((\Delta_1-\lambda_0\Delta_0)^H)=\{0\}$, which is similar to \eqref{cnd:ker_delta_I_V} and 
we prove it in an analogous way. 

For 2), we know from \cite[Thm. 4.5]{HMP_SingGep2}, see also lines 3 and 5 in Algorithm 2, that 
\begin{equation}\label{eq:proof_IU}
(I\otimes U_\perp)^H(\Delta_1-\lambda_0\Delta_0)(I\otimes V)(x_1\otimes x_2)=0.
\end{equation}
It follows from \eqref{eq:Delta_1_minus_Delta_0} and $(A+\lambda_0B)x_1=-\mu_0Cx_1$ that 
\eqref{eq:proof_IU} is equal to
\[
Cx_1\otimes U_\perp^H\left[\begin{matrix}A+\lambda_0 B +\mu_0 C & 0 \cr B & A+\lambda_0 B +\mu_0 C\end{matrix}\right]Vx_2=0.
\]
Since $C$ is generic, $Cx_1\ne 0$ and \eqref{projection_pog1} must hold. In a similar way, \eqref{projection_pog2} follows from the condition
\[
(y_1\otimes y_2)^H(I\otimes U)^H(\Delta_1-\lambda_0\Delta_0)(I\otimes V_\perp)=0.
\]
Finally, from \cite[Thm. 4.5]{HMP_SingGep2}, see also lines 4 and 5 in Algorithm 2, it follows that
\begin{align*}0&\ne (y_1\otimes y_2)^H(I\otimes U)^H\Delta_0(I\otimes V)(x_1\otimes x_2)\cr
&=y_1^HBx_1\cdot y_2^HU^H\left[\begin{matrix}C & \cr & C\end{matrix}\right]Vx_2
-y_1^HCx_1\cdot y_2^HU^H\left[\begin{matrix}B & \cr & B\end{matrix}\right]Vx_2,
\end{align*}
which is the condition \eqref{projection_pog3}.
\end{proof}

This leads to the following algorithm for the computation of 2D points. If we want to compute only ZGV points, we check for 
each obtained 2D point $(\lambda_0,\mu_0)$ if $\mu_0$ is a simple eigenvalue of the GEP \eqref{eq:gep_mu}.

\noindent\vrule height 0pt depth 0.5pt width \textwidth \\
{\bf Algorithm~3: Compute 2D points of pencil $A+\lambda B+\mu C$}. \\[-3mm]
\vrule height 0pt depth 0.3pt width \textwidth \\
{\bf Input:} $A,B,C\in\CC^{n\times n}$, thresholds $\delta_1,\delta_2>0$.\\
{\bf Output:} List of 2D points $(\lambda_i,\mu_i)$, $i=1,\ldots,r$.\\
\begin{tabular}{ll}
{\footnotesize 1:} & Select random unitary $2n\times 2n$ matrices
 $[U \ \, U_\perp]$ and 
 $[V \ \, V_\perp]$,
where $U$ and $V$ \\ & have $2n-1$ columns.\\
{\footnotesize 2:} & Compute the eigenvalues $(\lambda_i,\mu_i)$, $i=1,\dots,m$, and normalized right and left\\
&  eigenvectors $x_{i1}\otimes x_{i2}$ and $y_{i1}\otimes y_{i2}$ of the 2EP \eqref{eq:proj_2EP_ZGV}.\\
{\footnotesize 3:} & for $i=1,\ldots,m$:\\
{\footnotesize 4:} & \quad $\alpha_i=U_\perp^H\left[\begin{matrix}A+\lambda_i B +\mu_i C & 0 \cr B & A+\lambda_i B +\mu_i C\end{matrix}\right]Vx_{i2}$\\
{\footnotesize 5:} & \quad $\beta_i=y_{i2}^HU^H\left[\begin{matrix}A+\lambda_i B +\mu_i C & 0 \cr B & A+\lambda_i B +\mu_i C\end{matrix}\right]V_\perp$\\
{\footnotesize 6:} & \quad $\gamma_i=
y_{i1}^HBx_{i1}\cdot y_{i2}^HU^H\left[\begin{matrix}C & \cr & C\end{matrix}\right]Vx_{i2}
-y_{i1}^HCx_{i1}\cdot y_{i2}^HU^H\left[\begin{matrix}B & \cr & B\end{matrix}\right]Vx_{i2}$\\
{\footnotesize 7:} & \quad If ${\rm max}(\|\alpha_i\|_2,\|\beta_i\|_2) < \delta_1(\|A\|_2+|\lambda_i|\|B\|_2+|\mu_i|\|C\|_2)$ and\\
 & \quad 
$|\gamma_i|>\delta_2(1+|\lambda_i|^2)^{1/2}$, then add $(\lambda_i,\mu_i)$ to the list of 2D points.
\end{tabular} \\
\vrule height 0pt depth 0.5pt width \textwidth

In line 2 of Algorithm 3 we can compute all ($m=2n^2-n$) or just a small subset ($m\ll 2n^2-n$) of the eigenvalues. The first option can be applied to
problems of small size as direct methods for 2EP explicitly compute the $\Delta$-matrices 
and then apply the QZ algorithm. Compared to Algorithm~1, the $\Delta$-matrices in Algorithm~3 are of 
size $(2n^2-n)\times (2n^2-n)$, while the size in Algorithm~1 is $2n^2\times 2n^2$. The difference in size
is small, but more important is that the 2EP in Algorithm~3 is nonsingular, whereas the GEP in 
Algorithm~1 is singular.

The nonsingularity of the 2EP in Algorithm~3 opens up options for subspace methods, see, e.g., \cite{HKP_JD2EP, HP_HarmJDMEP, MP_SylvArnoldi}, that compute just a subset of the eigenvalues close to a given target.
These methods do not compute the $\Delta$-matrices explicitly and can thus be applied to larger problems.

\section{A zero-residual Gauss-Newton method for 2D points}\label{sec:Gauss_Newton}
If we are looking for 2D points of a bivariate pencil \eqref{eq:ABC}, then, by Definition~\ref{def:2D},
we are searching for scalars $\lambda,\mu\in \CC$ and vectors $x,y\in\CC^n$ that satisfy
\begin{equation}
\begin{split}\label{eq:sis_2D}
  (A+\lambda B+\mu C)x=0,\cr
  y^H(A+\lambda B+\mu C)=0,\cr
  y^HBx=0,\cr
  x^Hx=1,\cr
  y^Hy=1.\cr
\end{split}
\end{equation}
There is one equation more than the number of unknowns in $\lambda,\mu,x,y$, therefore
(\ref{eq:sis_2D}) is an overdetermined nonlinear system. In addition, it
is a zero-residual system since a 2D point and the corresponding right and left eigenvector solve 
\eqref{eq:sis_2D} exactly. 
The Gauss-Newton method can be applied to compute a 2D point from 
an initial approximation $(\lambda_0,\mu_0)$ for the 2D point and
$x_0, y_0$ for the right and left eigenvector.

As the equations in \eqref{eq:sis_2D} are not complex differentiable in $x$ and $y$, in order 
to get the Jacobian matrix, we
introduce random vectors $a$ and $b$ (we can assume that they are not orthogonal 
to eigenvectors $x$ and $y$, respectively) to replace the normalizing conditions,   
 define $w=\overline y$, and rewrite  \eqref{eq:sis_2D} as
\begin{equation}\label{eq:sis_2Dconj}
\begin{split}
  (A+\lambda B+\mu C)x=0,\cr
  (A^T+\lambda B^T+\mu C^T)w=0,\cr
  w^TBx=0,\cr
  a^Hx=1,\cr
  b^Hw=1.\cr
\end{split}
\end{equation}
Suppose that
we have an initial approximation $(\lambda_0,\mu_0,x_0,w_0)$ for the solution of \eqref{eq:sis_2Dconj}. Then we get a correction
$(\Delta\lambda_0,\Delta\mu_0,\Delta x_0,\Delta w_0)$ for the update 
$$(\lambda_1,\mu_1,x_1,w_1)=(\lambda_0,\mu_0,x_0,w_0)+(\Delta\lambda_0,\Delta\mu_0,\Delta x_0,\Delta w_0)$$
from the $(2n+3)\times (2n+2)$
least squares problem
$$J_F(\lambda_0,\mu_0,x_0,w_0)\Delta s_0=-F(\lambda_0,\mu_0,x_0,w_0),$$
where the Jacobian matrix is 
\begin{equation}\label{eq:jacobian}
J_F(\lambda_0,\mu_0,x_0,w_0) = \left[\begin{matrix} 
A+\lambda_0 B+\mu_0 C & 0 & Bx_0 & Cx_0 \cr
0 & A^T+\lambda_0 B^T+\mu_0 C^T & B^Tw_0 & C^Tw_0 \cr
w_0^TB & x_0^TB^T & 0& 0\cr
a^H & 0& 0& 0\cr
0& b^H &0 & 0\end{matrix}\right]
\end{equation}
and
$$\Delta s_0 = 
\left[\begin{matrix}\Delta x_0\cr \Delta w_0\cr \Delta\lambda_0\cr \Delta\mu_0\end{matrix}\right],
\quad
F(\lambda_0,\mu_0,x_0,w_0)=
\left[\begin{matrix}(A+\lambda_0 B+\mu_0 C)x_0\cr 
(A^T+\lambda_0 B^T+\mu_0 C^T)w_0\cr 
w_0^TBx_0 \cr a^Hx_0-1 \cr b^Hw_0-1\end{matrix}\right].
$$
The update $\Delta s_0$ is the least squares solution, i.e., 
$\Delta s_0=-J_F(\lambda_0,\mu_0,x_0,w_0)^+F(\lambda_0,\mu_0,x_0,w_0)$ using pseudoinverse of the Jacobian matrix.
The method is given in Algorithm 4.

\noindent\vrule height 0pt depth 0.5pt width \textwidth \\
{\bf Algorithm~4: Compute a 2D point of pencil $A+\lambda B+\mu C$}. \\[-3mm]
\vrule height 0pt depth 0.3pt width \textwidth \\
{\bf Input:} $A,B,C\in\CC^{n\times n}$, initial approximations $(\lambda_0,\mu_0)$ for the 2D point and
$x_0, y_0$ for the right and left eigenvector, random vectors $a,b$\\
{\bf Output:} A 2D point $(\lambda,\mu)$\\
\begin{tabular}{ll}
{\footnotesize 1:} & Set $w_0=\overline y_0$\\
{\footnotesize 2:} & for $k=0,1,2,\ldots$ until convergence do\\
{\footnotesize 3:} & \quad Compute\\ 
& \qquad $J_F = \left[\begin{matrix} 
A+\lambda_k B+\mu_k C & 0 & Bx_k & Cx_k \cr
0 & A^T+\lambda_k B^T+\mu_k C^T & B^Tw_k & C^Tw_k \cr
w_k^TB & x_k^TB^T & 0& 0\cr
a^H & 0& 0& 0\cr
0& b^H &0 & 0\end{matrix}\right],$\\[3em]
{\footnotesize 4:} & \quad $F=
\left[\begin{matrix}(A+\lambda_k B+\mu_k C)x_k\cr 
(A^T+\lambda_k B^T+\mu_k C^T)w_k\cr 
w_k^TBx_k \cr a^Hx_k-1 \cr b^Hw_k-1\end{matrix}\right],
\qquad
\left[\begin{matrix}\Delta x_k\cr \Delta w_k\cr \Delta\lambda_k\cr \Delta\mu_k\end{matrix}\right]
= -J_F^+ F
$\\[2em]
{\footnotesize 5:} & \quad  $x_{k+1}=x_k+\Delta x_k$, 
$w_{k+1}=w_k+\Delta w_k$, $\lambda_{k+1}=\lambda_k+\Delta \lambda_k$,
 $\mu_{k+1}=\mu_k+\Delta \mu_k$
\end{tabular} \\
\vrule height 0pt depth 0.5pt width \textwidth
\medskip

A numerical method 
that provides good initial 
approximations $(\lambda_0,\mu_0)$ for 2D points is the MFRD from Section~\ref{sec:elias}. Beside an  approximation for a 2D point, Algorithm 4 requires initial approximations for 
the right and left eigenvector as well. 
If they are not provided, we use the following heuristics to set the initial vectors $x_0$ and $y_0$.
Let $U\Sigma V^T$ be a singular value decomposition of $A+\lambda_0 B+\mu_0 C$. 
If $\sigma_{n-1}\ll \sigma_{n-2}$ or $\sigma_{n-1}\approx \sigma_n$, this suggest that a possible 2D point
is of type c) or d). In this case we take for $x_0$ a random combination of the right singular vectors $v_{n-1}$ and $v_n$ such that $\|x_0\|_2=1$ and then select for
$y_0$ a linear combination of $u_{n-1}$ and $u_n$ such that $y_0^HBx_0=0$ and $\|y_0\|_2=1$. Otherwise, we
assume that a 2D point is of type a) or b), and take $x_0=v_n$ and $y_0=u_n$.

In general the convergence of the Gauss-Newton method is not guaranteed. It is known however that the Gauss-Newton method  converges locally quadratically for a zero-residual problem if the 
Jacobian $J_F$ has full rank at the solution, see, e.g., \cite[Section 4.3.2]{Deuflhard_2011} or 
\cite[Section 10.4]{Nocedal_Wright_2006}. We can show that the Jacobian $J_F(\lambda_0,\mu_0,x_0,y_0)$ has full rank
at a generic 2D point of type a), where $\lambda_0$ is a double eigenvalue of \eqref{eq:gep_lambda},
$(\lambda_0,\mu_0)$ is
a ZGV point and $x_0$ and $y_0$ are the corresponding right and left eigenvector. To show this, we need the following auxiliary result.

\begin{lemma}\label{lem:jacobi_fullrank}
    Let $\xi\in\CC$ be 
    an eigenvalue of algebraic multiplicity two and geometric multiplicity one of 
    an $n\times n$ regular matrix pencil $A-\lambda B$. Let nonzero vectors $x,y,s,t\in\CC^n$ be respectively 
    the right and the left eigenvector and the left and the right 
    generalized eigenvectors of order two such that
\begin{equation}\label{eq:yBs_tBx}
\begin{split}
        (A-\xi B)x&=0,\\
        (A-\xi B)s&=Bx,\\
        y^H(A-\xi B)&=0,\\
        t^H(A-\xi B)&=y^HB.
    \end{split}
\end{equation}    
    Then, $y^HBs=t^HBx\ne 0$.
\end{lemma}
\begin{proof}
    The equality $y^HBs=t^HBx$ follows if we multiply the second equation of \eqref{eq:yBs_tBx} by $t^H$ from the left
    and the fourth equation by $s$ from the right. 
    
    To show that the quantity is nonzero, we use the 
    Kronecker canonical form (KCF) of the pencil $A-\lambda B$, see, e.g., \cite{Gantmacher2}.
    If follows from the KCF of $A-\lambda B$ that there exist nonsingular matrices $P$ and $Q$ such that
    \[
    P(A-\lambda B)Q = 
    \left[\begin{matrix} J_2(\xi) & \cr
        & C-\lambda D
    \end{matrix}\right],\quad {\rm where}\ J_2(\xi)=\left[\begin{matrix}
        \xi-\lambda & 1\cr & \xi-\lambda
    \end{matrix}\right]
    \]
    and $C-\lambda D$ is a regular $(n-2)\times (n-2)$ block diagonal pencil that contains the remaining Jordan  
    and infinite blocks and 
    such that $\xi$ is not an eigenvalue of 
    $C-\lambda D$. It is easy to see that 
    \begin{align*}
    x&=\alpha_1Qe_1,\quad s=\alpha_1Qe_2+\alpha_2 Qe_1,\cr   
    y&=\beta_1P^{H}e_2,\quad t=\beta_1P^{H}e_1+\beta_2 P^{H}e_2
    \end{align*}
    for some scalars $\alpha_1,\alpha_2,\beta_1,\beta_2$, where $\alpha_1$ and $\beta_1$ are nonzero. It then follows from
    the second equation of \eqref{eq:yBs_tBx} that 
    \[ t^HBx = t^H(A-\xi B)s = (\beta_1e_1+\beta_2e_2)^H
    \left[\begin{matrix}
        0 & 1 & \cr & 0 & \cr & & C-\xi D
    \end{matrix}\right] (\alpha_1e_2+\alpha_2e_1)= \overline \beta_1 \alpha_1 \ne 0.
    \]
\end{proof}

\begin{lemma}\label{lem:root_eigv}
Let $(\lambda_0,\mu_0)\in\CC^2$ be a ZGV point of a biregular bivariate pencil $A+\lambda B +\mu C$, let $x_0$ and $y_0$ be the corresponding right and left eigenvector, and $w_0=\overline y_0$. If the algebraic multiplicity
of $\lambda_0$ as an eigenvalue of \eqref{eq:gep_lambda} is two, then the 
Jacobian \eqref{eq:jacobian}, where $a$ and $b$ are such nonzero vectors that $a^Hx_0=1$ and $b^Hw_0=1$, has full rank.    
\end{lemma}

\begin{proof}
    Suppose that the Jacobian $J_F(\lambda_0,\mu_0,x_0,w_0)$ is rank deficient. Then there exist vectors 
    $s,t$ and scalars $\alpha,\beta$, not all being zero, such that
    $J_F(\lambda_0,\mu_0,x_0,w_0)\left[\begin{matrix}
        s^T & t^T & \alpha & \beta
    \end{matrix}\right]^T=0,
    $
    i.e.,
\begin{align}
  (A+\lambda_0 B+\mu_0 C)s +\alpha Bx_0 + \beta Cx_0&=0,\label{eq:dokaz1}\\
  (A^T+\lambda B^T+\mu C^T)t+\alpha B^T w_0 + \beta C^Tw_0&=0,\label{eq:dokaz2}\\
  w_0^TBs + x_0^TB^Tt&=0,\label{eq:dokaz3}\\
  a^Hs&=0,\label{eq:dokaz4}\\
  b^Ht&=0\label{eq:dokaz5}.
\end{align}
If we multiply \eqref{eq:dokaz1} from the left by $y_0^H$ then 
Lemma~\ref{lem:2D} yields that $\beta=0$ as $(\lambda_0,\mu_0)$ is a ZGV point. 

If $\alpha\ne 0$, then it follows from 
\eqref{eq:dokaz1} and \eqref{eq:dokaz2} that $s$ and $\overline t$ are (up to a multiplication by a nonzero scalar) left and right generalized eigenvectors of degree two of the GEP \eqref{eq:gep_lambda} for the eigenvalue $\lambda_0$. But then by
Lemma~\ref{lem:root_eigv} $w_0^TBs=x_0^TB^Tt\ne 0$ and \eqref{eq:dokaz3} does not hold. Therefore, $\alpha=0$.

Since $\alpha= \beta=0$ it follows from \eqref{eq:dokaz1} that $s=\gamma x_0$ for a scalar $\gamma$ and then $s=0$
because of \eqref{eq:dokaz4}. In a similar way we get from \eqref{eq:dokaz2} and \eqref{eq:dokaz5} that
$t=0$. This shows that there are no nonzero vectors in the kernel of $J_F(\lambda_0,\mu_0,x_0,w_0)$.
\end{proof}

A generic ZGV point $(\lambda_0,\mu_0)$ is such that $\lambda_0$ is a double eigenvalue of \eqref{eq:gep_lambda}. 
It follows from Lemma~\ref{lem:jacobi_fullrank} that near
such points Algorithm~4 converges quadratically. There might also exist ZGV points where multiplicity of $\lambda_0$ is higher than two. At such points we can expect a linear convergence.

Although we are primarily interested in ZGV points, Algorithm~4 converges to all 
types of 2D points for a bivariate pencil $A+\lambda B+\mu C$. If a 2D point is not a ZGV point, i.e., it is a 2D point of type b), c) or d), then the Jacobian \eqref{eq:jacobian} is rank
deficient and we can expect a linear convergence.

\section{Method of fixed relative distance}\label{sec:elias}
A method from \cite{Elias_DoubleEig}, designed
to compute approximations for values $\mu$ such that $A+\mu B$ has a multiple eigenvalue, can be straightforward generalized 
to compute good approximations for 2D points of a biregular pencil
\eqref{eq:ABC}, as already discussed in \cite[Sec.~6]{Elias_DoubleEig}.  

By Lemma~\ref{lem:dbl} $(\lambda_0,\mu_0)\in\CC^2$ is a 2D point of \eqref{eq:ABC} if and only if $\lambda_0$ is a multiple (generically double) eigenvalue of \eqref{eq:gep_lambda}. Thus, for $\widetilde\mu_0\approx \mu_0$ such that
$\widetilde\mu_0\ne \mu_0$ the GEP
\begin{equation}
\label{eq:pert_elias}\left((A+\widetilde \mu_0 C)+\lambda B\right)x=0
\end{equation}
has two different eigenvalues $\lambda_{01}$ and $\lambda_{02}$ close to $\lambda_0$. 
In the method of fixed relative distance (MFRD) we select a regularization parameter $\delta>0$ and
search for $\widetilde \mu_0$ such that \eqref{eq:pert_elias} has eigenvalues $\lambda_{01}$ and $\lambda_{02}=(1+\delta)\lambda_{01}$. We 
write this as a 2EP
\begin{equation}\label{eq:2ep_elias}
\begin{split}
  (A+\lambda B+\mu C)x_1&=0,\\
  (A+\lambda(1+\delta) B+\mu C)x_2&=0.
\end{split}
\end{equation}
The 2EP \eqref{eq:2ep_elias} is generically nonsingular for $\delta>0$ and eigenvalues are approximations of 2D points of \eqref{eq:ABC}. Note that
\eqref{eq:2ep_elias} has $n^2$ eigenvalues while generically there are only $n(n-1)$ 2D points of \eqref{eq:ABC}, so
at least $n$ of the 
eigenvalues of \eqref{eq:2ep_elias} are not related to 2D points of \eqref{eq:ABC}. It is easy to see that
$\eqref{eq:2ep_elias}$ has $n$ solutions of the form $(0,\mu)$, where $\mu$ is an eigenvalue of the GEP
$(A+\mu C)x=0$, which are generically not related to nearby 2D points.
For each eigenvalue
$(\lambda,\mu)$ of \eqref{eq:2ep_elias} we thus have to check if it corresponds to a 2D point of \eqref{eq:ABC}. For this task and also to refine
the solution, which is inaccurate due to the regularization parameter $\delta$, we use Algorithm~4.
A parallel approach is applied in \cite{Elias_DoubleEig}, where approximations obtained from a regularized 2EP
are  refined using a similar zero-residual Gauss-Newton method. 

The role of the regularization parameter $\delta$ is the following. If $\delta$ is too small, then the GEP \eqref{eq:2ep_elias} is close to being singular which makes it difficult to find eigenvalues using methods for nonsingular 2EPs. 
On the other hand, by increasing the value of $\delta$ approximations for the 2D points are becoming less accurate and Algorithm~4 might fail
to converge to a 2D point.

\section{Numerical examples and applications}\label{sec:applications}
We provide several numerical examples together with some applications and related problems. 
All numerical experiments were carried out in Matlab
2023a on a desktop PC with 64 GB RAM and i7--11700K 3.6 GHz CPU. Methods from MultiParEig \cite{multipareig_28} were used to solve the related 2EPs and singular GEPs. Some results were computed 
in higher precision using the Advanpix Multiprecision Computing Toolbox~\cite{advanpix}. 
The code and data for all
numerical examples in this paper are available at 
\url{https://github.com/borplestenjak/ZGV_Points}.

\begin{example}\label{ex:ex2x2}\rm We start with a simple 
example, where we consider
$$A=\left[\begin{matrix}3 & 0  \cr 0 & 0\end{matrix}\right],\quad
  B=\left[\begin{matrix}0 & 1\cr -1 & -1\end{matrix}\right],\quad
  C=\left[\begin{matrix}-2 & -2 \cr 2 & 0\end{matrix}\right].\quad
$$    
Then $\det(A+\lambda B+\mu C)=\lambda^2-2\lambda\mu+4\mu^2-3\lambda$ and real eigenvalues $(\lambda,\mu)$ lie on the ellipse 
shown in Figure \ref{fig:ex2x2}. The two ZGV points, marked on the figure, are $Z_1=(1,0.5)$ and $Z_2=(3,1.5)$.

\begin{figure}[h]
    \centering
    \includegraphics[width=6cm]{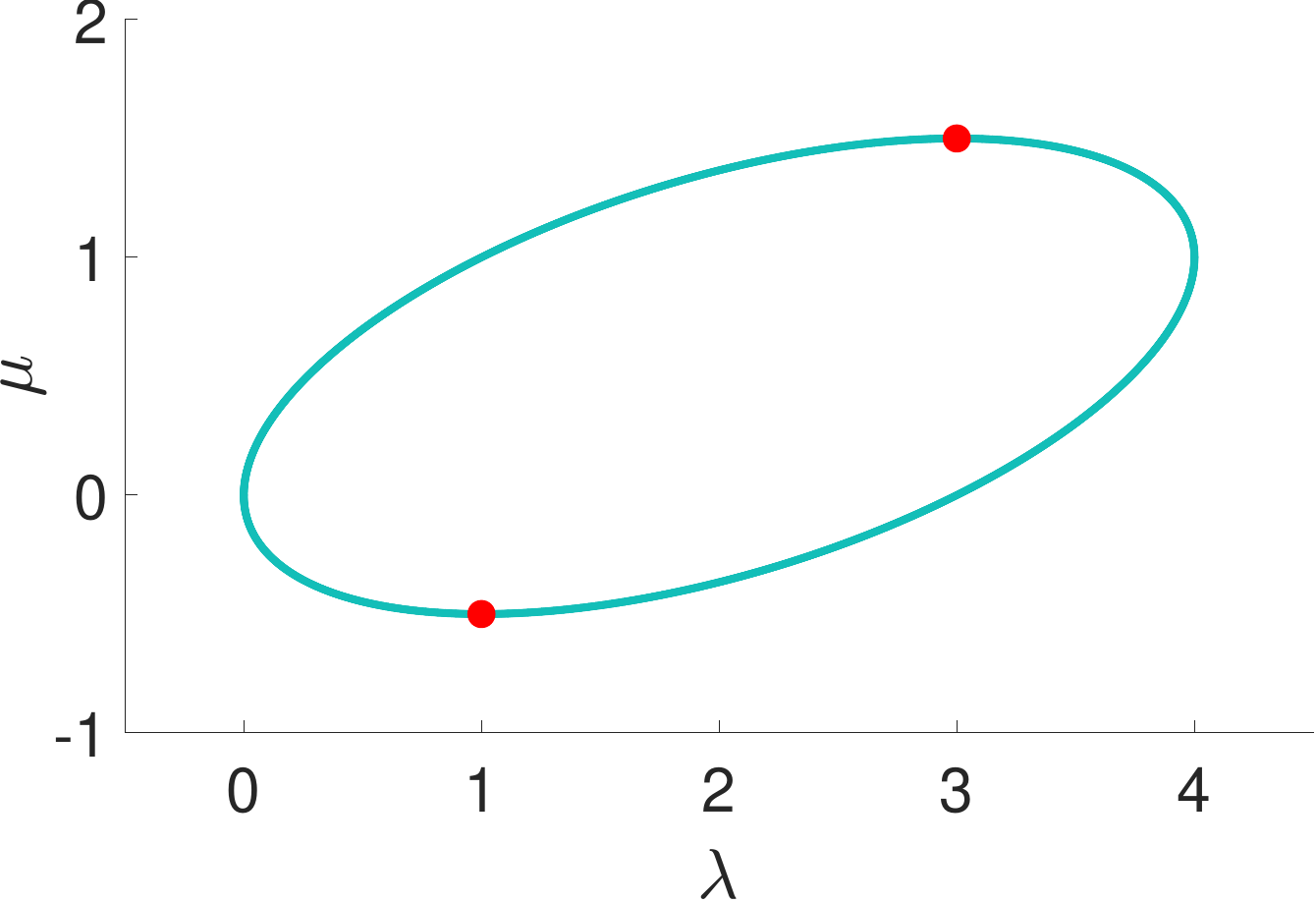}
    \caption{Real eigenvalues and ZGV points of Example \ref{ex:ex2x2}.}
    \label{fig:ex2x2}
\end{figure}

If we apply Algorithm~1, we get matrices $\Delta_0,\Delta_1,\Delta_2$ \eqref{Deltaik} of size $8\times 8$. The singular pencil
$\Delta_1-\lambda \Delta_0$ has two finite eigenvalues $\lambda_1=1$ and $\lambda_2=3$, which we 
compute using Algorithm~2.
By computing the corresponding eigenvalues 
$\mu$ of 
the pencils $(A+\lambda_i B)+\mu C$ we get both ZGV points. The maximal error 
$\|Z_i-(\lambda_i,\mu_i)\|_2$ for $i=1,2$ of the computed ZGV points is $1.6\cdot 10^{-15}$.

Alternatively, we could use Algorithm~3. The eigenvalues $(\lambda_i,\mu_i)$ of the projected regular 
2EP are listed in Table \ref{tab:2x2} together with the values $\widetilde\alpha_j=\alpha_j/(\|A\|_2+|\lambda_j|\|B\|_2+|\mu_j|\|C\|_2)$,
$\widetilde\beta_j=\beta_j/(\|A\|_2+|\lambda_j|\|B\|_2+|\mu_j|\|C\|_2)$, and
$\widetilde \gamma_j = \gamma_j(1+|\lambda_j|^2)^{-1/2}$ that are used to identify the two regular
eigenvalues corresponding to ZGV points. The maximal error in this case is $3.0\cdot 10^{-15}$.

\begin{table}[htb!] 
\centering
\caption{Results of Algorithm~3 applied to 
Example \ref{ex:ex2x2}.}\label{tab:2x2}
\begin{tabular}{c|cclll} \hline \rule{0pt}{2.85ex}%
$j$ & $\lambda_j$ & $\mu_j$ & $\quad \ \widetilde\alpha_j$ & $\quad \ \widetilde\beta_j$ & $\quad \ \widetilde\gamma_j$  \\[0.5mm]
\hline \rule{0pt}{2.5ex}%
1 & $1.00000 - 1.9\cdot 10^{-15}i$ & $-0.50000 - 1.2\cdot 10^{-15}i$ & $3.9\cdot 10^{-16}$ & $1.6\cdot 10^{-16}$ & $7.9\cdot 10^{-2}$ \\
2 & $3.00000 - 2.3\cdot 10^{-15}i$ & $ \phantom{-}1.50000 - 6.1\cdot 10^{-16}i$ & $8.9\cdot 10^{-17}$ & $2.5\cdot 10^{-16}$ & $3.3\cdot 10^{-1}$ \\
3 & $1.4\cdot 10^{15} - 1.1\cdot 10^{15}i$ & $ \phantom{-}8.3\cdot 10^{14} + 3.1\cdot 10^{14}i$ & $7.7\cdot 10^{-16}$ & $3.8\cdot 10^{-16}$ & $1.3\cdot 10^{-31}$ \\
4 & $2.9\cdot 10^{15} + 2.9\cdot 10^{15}i$ & $\phantom{-}2.0\cdot 10^{15} - 5.3\cdot 10^{14}i$ & $1.2\cdot 10^{-16}$ & $6.6\cdot 10^{-17}$ & $3.0\cdot 10^{-32}$ \\
5 & $0.96018 + 0.34932i$ & $-0.52033-0.00224i$ & $4.1\cdot 10^{-2}$ & $3.8\cdot 10^{-16}$ & $8.2\cdot 10^{-2}$ \\
6 & $4.18091-0.52938i$ & $\phantom{-}1.46522-0.38308i$ & $5.8\cdot 10^{-16}$ & $7.3\cdot 10^{-2}$ & $1.2\cdot 10^{-1}$ \\
\hline
\end{tabular}
\end{table}

Yet another option is to apply the MFRD and Algorithm~4. If we take $\delta=10^{-2}$ and solve the
2EP \eqref{eq:2ep_elias}, we get four candidates for ZGV points:
$(\widetilde\lambda_1,\widetilde\mu_1)=(0.99503,-0.49999)$,
$(\widetilde\lambda_2,\widetilde\mu_2)=(2.98504,1.49996)$,
$(\widetilde\lambda_3,\widetilde\mu_3)=(-1.24\cdot 10^{-22},1.63\cdot 10^{-11})$,
$(\widetilde\lambda_4,\widetilde\mu_4)=(0,0)$.
We use them as initial values for the Gauss-Newton method in
Algorithm~4 and select initial approximations for left and right eigenvectors from
the SVD of $A+\widetilde \lambda_i B+\widetilde \mu_i C$. The first two candidates converge
quadratically to $Z_1$ and $Z_2$, respectively, with the maximal error $1.2\cdot 10^{-16}$. 
\end{example}

\begin{example}\label{ex:conv_plot}\rm To illustrate the quadratic convergence of Algorithm 4 near
a ZGV point, we apply the algorithm to the pencil \eqref{ex:mat_exaone} from Example~\ref{exa:one} and initial approximations obtained by the MFRD using $\delta=10^{-2}$. The convergence plots of nine obtained solutions, composed of six ZGV points and three 2D points of type d), where the eigencurves intersect, are presented in Figure \ref{fig:convg_plot}. To emphasize the quadratic convergence, 
we used computation in higher precision. 
Although the  Jacobian is singular at the solution, we observe initial
quadratic convergence for the intersection points as well, followed by steps where the error suddenly increases close to the solution. This happens because for these eigenvalues the right and left eigenvectors are not unique. 

\begin{figure}[h]
    \centering
    \includegraphics[width=6cm]{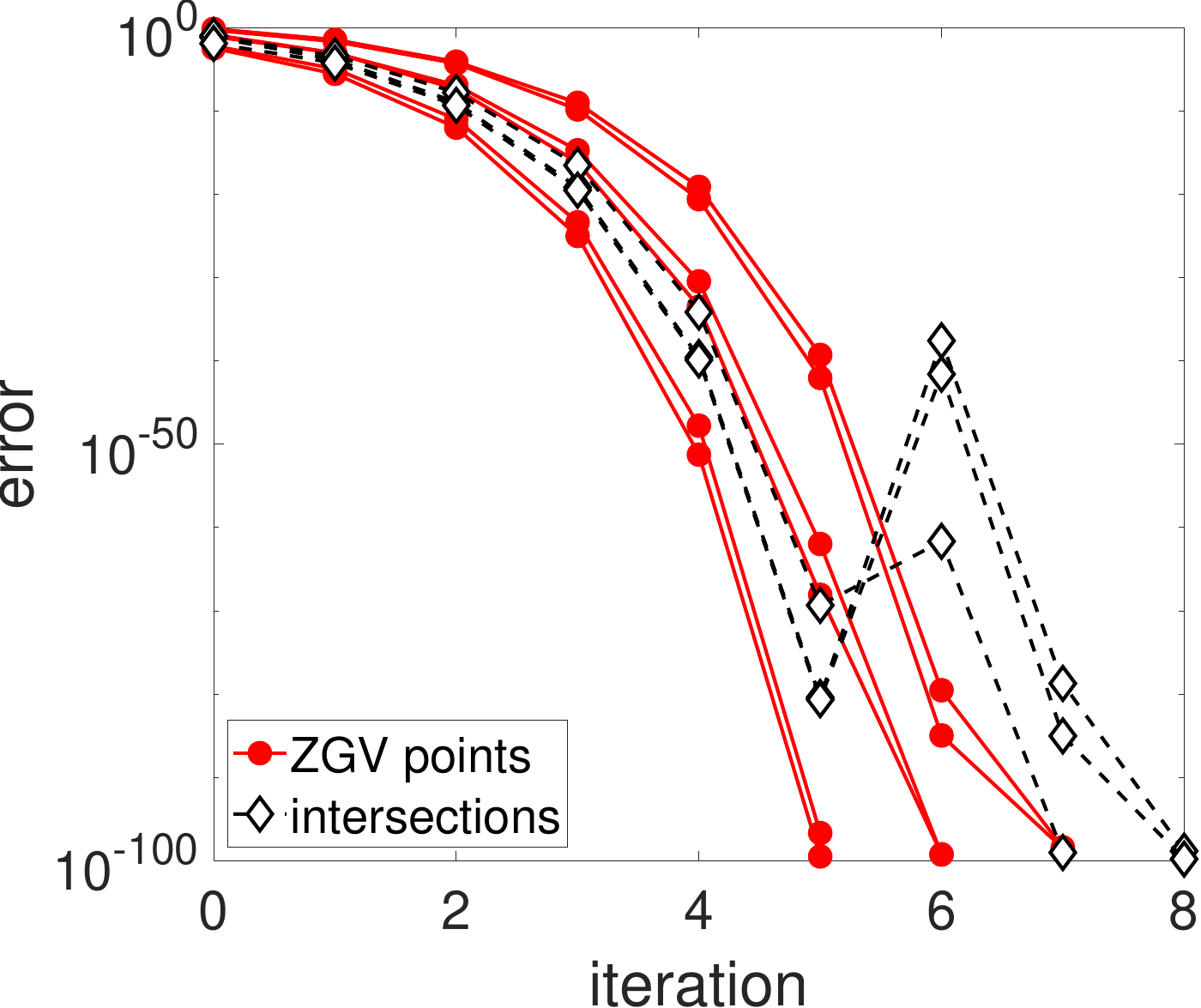}
    \caption{Convergence plot for Algorithm 4 applied to 2D points of Example \ref{exa:one}.}
    \label{fig:convg_plot}
\end{figure}
\end{example}

\begin{example}\label{ex:rand5_45}\rm We take random
real $n\times n$ matrices $A$, $B$, $C$ and compare the timings and errors of the computed ZGV points.
The exact ZGV points to compare with were computed by an additional refinement using Algorithm 4 in 
quadruple precision.
For each $n=5,10,\ldots,45$ we solved 10 random problems. The comparison of median computational time and median 
maximal relative error of the obtained ZGV points is presented in Figure \ref{fig:rand5_45}. We applied Algorithm 4 to refine the results obtained by Algorithm 1 and Algorithm 3 as well.

\begin{figure}[h]
    \centering
    \includegraphics[width=6cm]{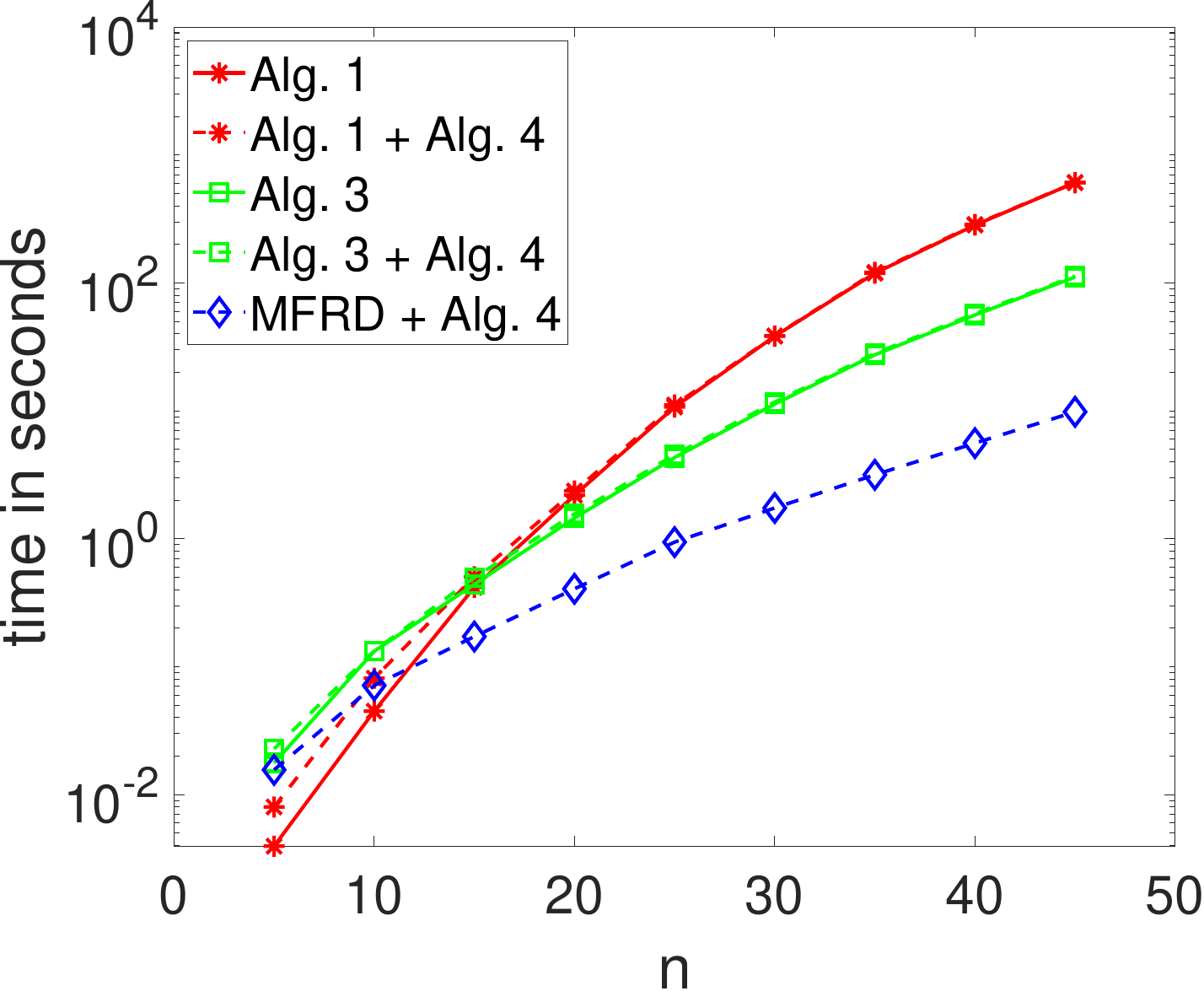}\quad  \includegraphics[width=6cm]{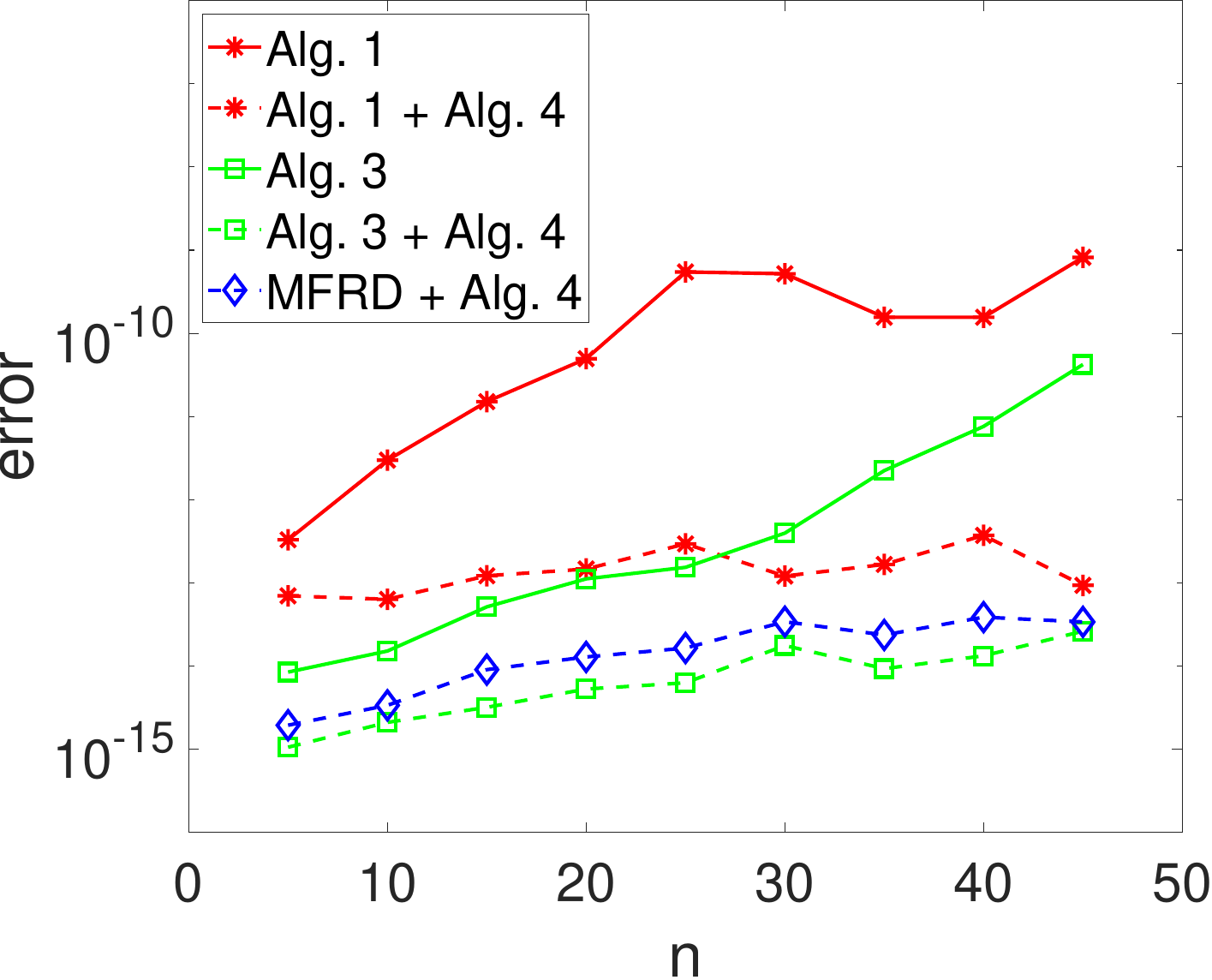}
    \caption{Computational times (left hand side) and errors (right hand side) on random pencils.}
    \label{fig:rand5_45}
\end{figure}
\end{example}

All three methods have complexity ${\cal O}(n^6)$. Algorithm 1 is the most expensive 
since it requires the solution of a singular GEP of size $2n^2\times 2n^2$ to compute the candidates for $\lambda$. In addition, for each candidate we have to solve an $n\times n$ GEP to obtain the possible $\mu$ parts. 
Algorithm 3 is more efficient as we have to solve a regular and slightly smaller GEP of size $(2n^2-n)\times (2n^2-n)$, and no additional smaller GEPs have to be solved. The MFRD is asymptotically the fastest approach as the main task is 
to solve a regular MEP with matrices of size $n\times n$, which is equivalent to solving a regular GEP of size $n^2\times n^2$.

As expected, the MFRD combined with Algorithm 4 is the fastest method for large $n$. It is also the most accurate method due to the final refinement by the Gauss-Newton method, but note that this refinement can be applied with virtually no extra cost to
Algorithm 1 and Algorithm 3 as well which reduces the error substantially. Based on these results we see that Algorithms 1 and 3 are suitable for small problems where $n\le 10$, while for larger problems one should
use the MFRD. 

\subsection{2D-eigenvalue problem}\label{sec:2D}

In \cite{LuSuBai_SIMAX} a 2D-eigenvalue problem (2DEVP) is studied, which turns
out to be a special case of 2D points. In 2DEVP we have Hermitian matrices
$A, B\in\CC^{n\times n}$, where $ B$ is indefinite, and we are searching for a
pair $(\lambda,\mu)\in\RR^2$, called a \emph{2D-eigenvalue}, 
and nonzero $x\in\CC^n$ such that
\begin{equation}\label{eq:2DEVP}
\begin{split}
    (A-\lambda B)x&=\mu x,\cr
    x^H Bx &= 0,\cr
    x^Hx & = 1.
\end{split}
\end{equation}
We can see (\ref{eq:2DEVP}) as a special case of (\ref{eq:sis_2D}).
Namely, for $\lambda\in\RR$ is the matrix $A-\lambda B$ Hermitian and has real eigenvalues 
$\mu_1,\ldots,\mu_n$. The corresponding left and right eigenvectors are equal.
Thus, each 2D-eigenvalue of \eqref{eq:2DEVP} is a 2D point of the bivariate pencil 
$A+\lambda (-B)+\mu (-I)$
and each solution $(\lambda,\mu,x)$ of  (\ref{eq:2DEVP}) gives a solution 
$(\lambda,\mu,x,x)$ of the corresponding (\ref{eq:sis_2D}).
The related problem \eqref{eq:sis_2D}
can have additional solutions as $\lambda$ and $\mu$ in (\ref{eq:sis_2D}) can 
be complex and right and left eigenvectors $x$ and $y$ are not necessary colinear. 
 Therefore, we can apply Algorithm~1 or Algorithm~3 to compute
all solutions of (\ref{eq:sis_2D}) and then a subset of real pairs $(\lambda,\mu)$ is the
solution of (\ref{eq:2DEVP}). Up to our knowledge, this is the first global method
that computes 
all 2D-eigenvalues of a 2DEVP. All other algorithms, including the one introduced in
\cite{LuSuBai_SIMAX}, are local methods that require good initial approximations 
and compute a single solution of (\ref{eq:2DEVP}). 

Alternatively, we can apply the MFRD from Section~\ref{sec:elias} to obtain good approximations for the solutions
of a related problem \eqref{eq:sis_2D} and then refine the approximations by applying the Gauss-Newton type method to \eqref{eq:2DEVP} in
a similar way as in Algorithm~4.

\begin{example}\label{ex:lusu} \rm The matrices of a 2D-eigenvalue problem from \cite[Ex.~1]{LuSuBai_SIMAX} are
$$A=\left[\begin{matrix}2 & 0  &1\cr 0 & 0 & 1\cr 1 & 1 & 0\end{matrix}\right],\quad
  B=\left[\begin{matrix}1 & 0  &1\cr 0 & 1 & 1\cr 1 & 1 & 0\end{matrix}\right].\quad
$$    
Figure~\ref{fig:LuSuBai} shows the real eigencurves of the corresponding bivariate pencil 
$A+\lambda (-B)+\mu (-I)$ together with the three 2D-eigenvalues, which are all 2D points
of type a). The 2D point $(1,0)$ is such that $\lambda=1$ is an eigenvalue of multiplicity three for 
the pencil $A-\lambda B$.

\begin{figure}[h]
    \centering
    \includegraphics[width=6cm]{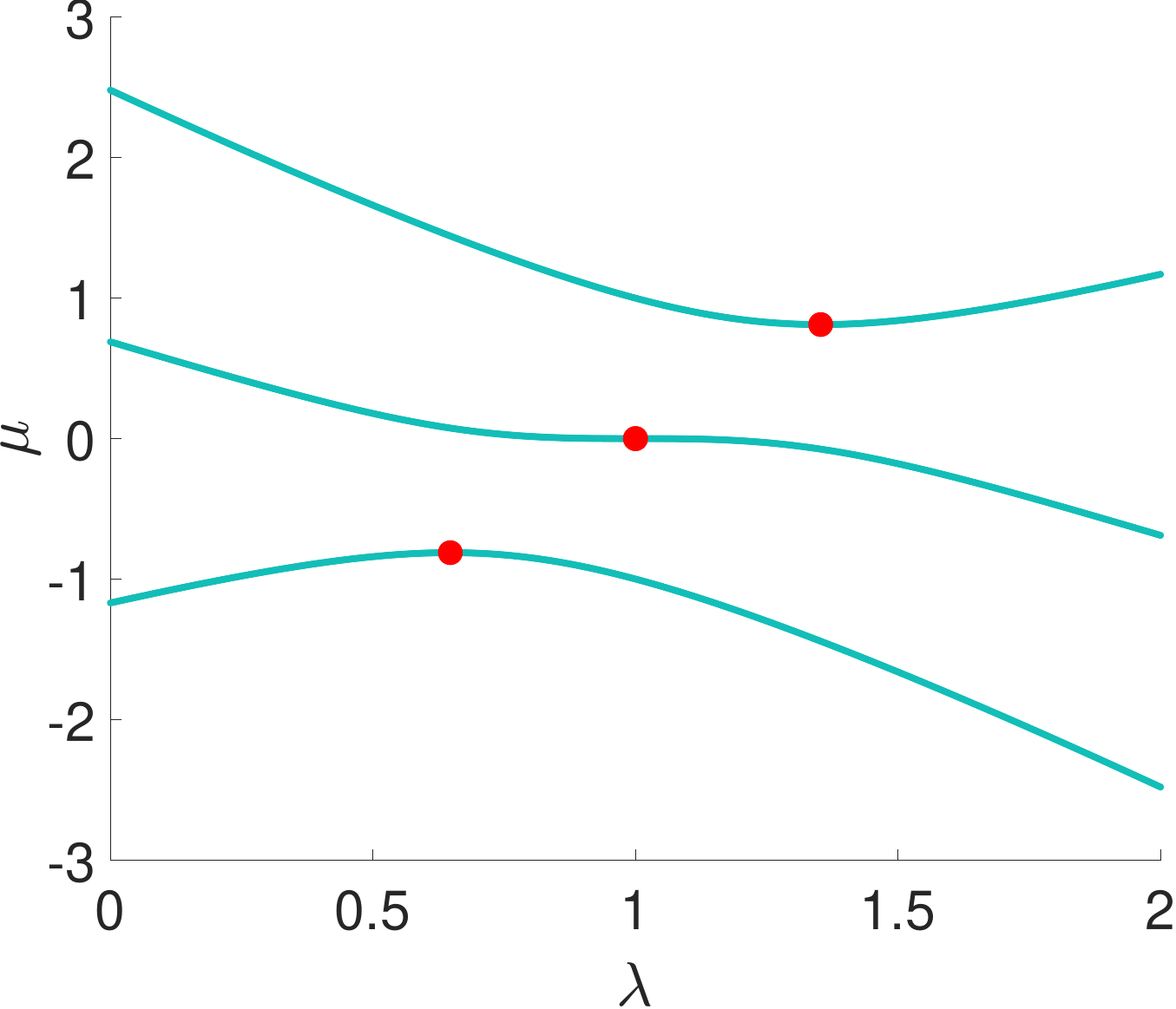}
    \caption{2D-eigenvalues of Example \ref{ex:lusu}}
    \label{fig:LuSuBai}
\end{figure}
\end{example}

We apply Algorithm~1 and construct matrices $\Delta_0,\Delta_1,\Delta_2$ \eqref{Deltaik} of size $18\times 18$. The singular pencil
$\Delta_1-\lambda \Delta_0$ has six finite eigenvalues $\lambda_1,\ldots,\lambda_6$. By computing the corresponding eigenvalues 
$\mu$ of 
the pencils $(A-\lambda_i B)-\mu I$ we obtain six 2D points:
$(1 \pm 1.5\cdot 10^{-8},\mp\, 6.3\cdot 10^{-24})$, $(1.3527,0.8121)$,
$(0.6473,-0.8121)$, and $(1.0000\pm 1.6371i,1.1\cdot 10^{-15}\mp 2.1327i)$. The point $(1,0)$ is computed twice because
$\lambda=1$ is a double eigenvalue of $\Delta_1-\lambda \Delta_0$.

\begin{example}\label{ex:big} \rm 
We consider 2D-eigenvalues of $n\times n$ banded Toeplitz matrices 
\begin{equation}\label{eq:penta}
A={\rm pentadiag}(1,0,5,0,1),\quad B={\rm tridiag}(1,1/2,1).
\end{equation} 
For each $\lambda\in\RR$ the 
eigenvalues of $A-\lambda B$ can be ordered as $\mu_n(\lambda)\le \cdots\le \mu_1(\lambda)$.
The real eigencurves
of $A-\lambda B-\mu I$ have a particular structure that is visible in Figure~\ref{fig:ExamplePenta} for $n=10$ (left hand figure) and  $n=20$ (close up, right hand figure). For $n=2m$ the curves 
$\mu_{2k-1}(\lambda)$ and $\mu_{2k}(\lambda)$ for $k=1,\ldots,m$ touch at $2k-1$ points, which are
2D points of type d) and are represented by red dots. Together there are $m^2$ such 2D-eigenvalues. The additional 2D-eigenvalues are
ZGV points represented by white dots, where $\mu'(\lambda)=0$. The 2D-eigenvalues in
Figure~\ref{fig:ExamplePenta} were computed with the MFRD combined with Algorithm 4.

\begin{figure}[h]
    \centering
    \includegraphics[width=6cm]{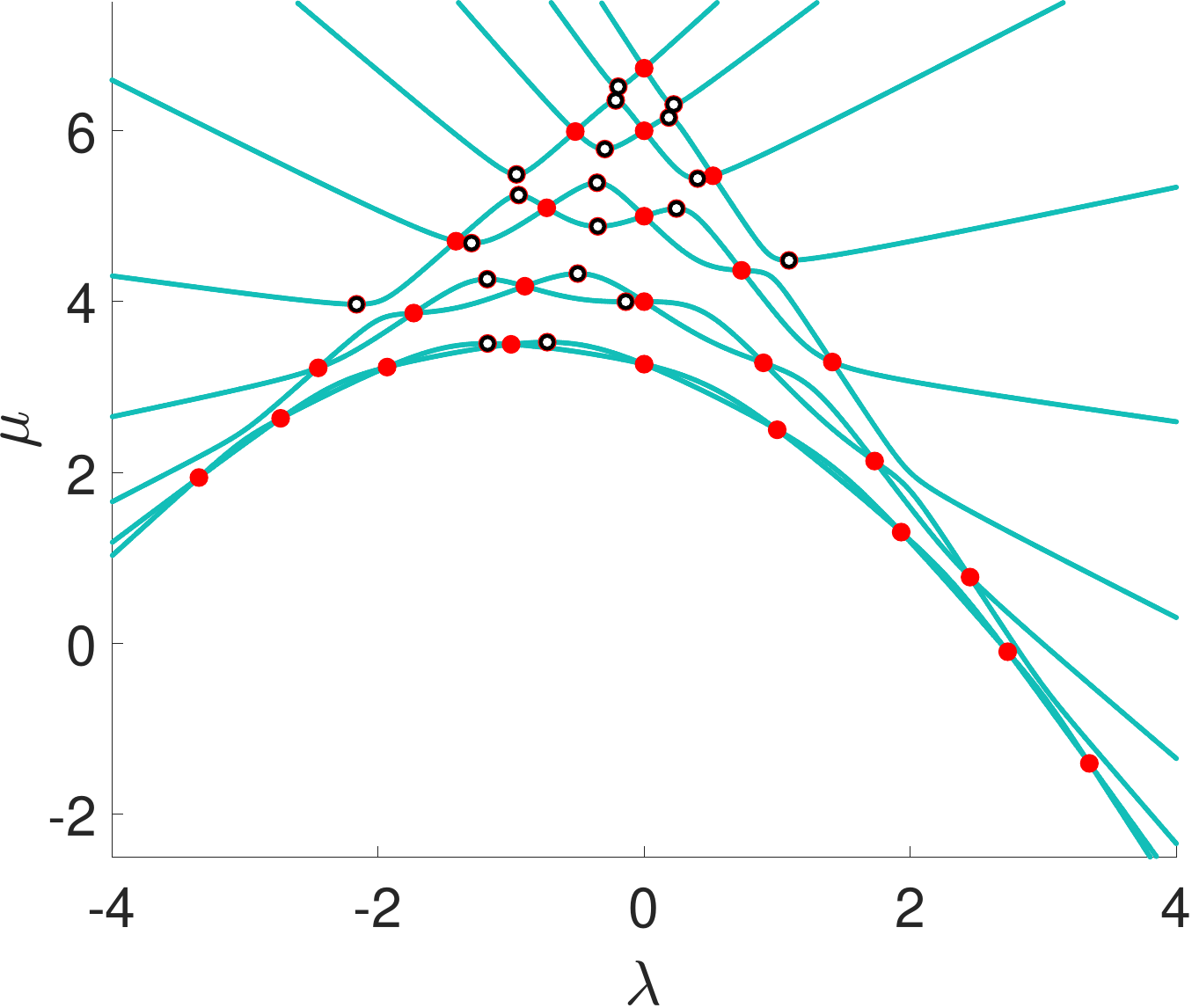}\quad
    \includegraphics[width=6cm]{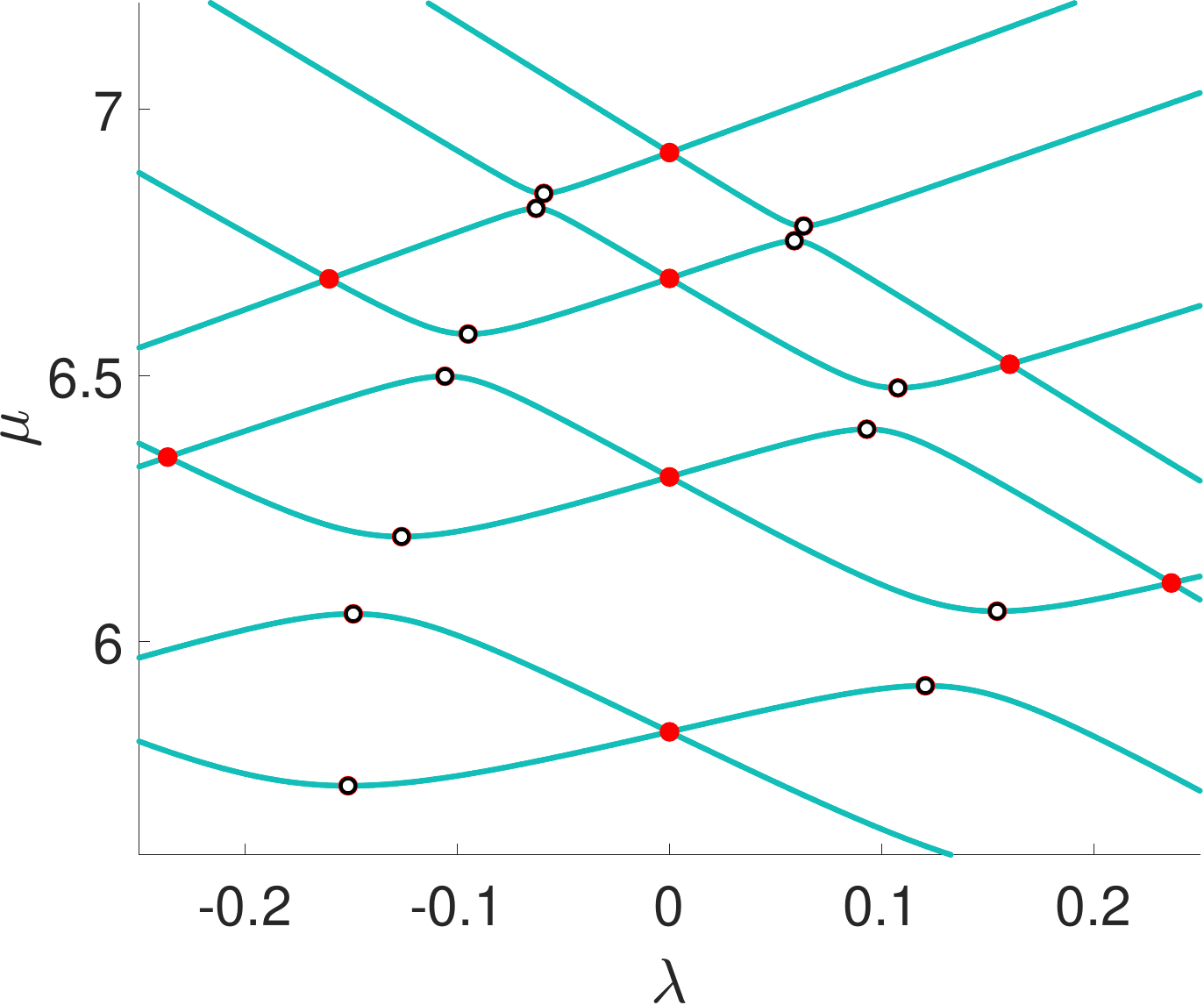}
    \caption{2D-eigenvalues of \eqref{eq:penta} for $n=10$ (left hand side) and 
    $n=20$ (right hand side), where white and red points respectively represent ZGV points and 2D points of type d).}
    \label{fig:ExamplePenta}
\end{figure}

For a large $n$ it is not possible to compute all 2D-eigenvalues due to the size of the $\Delta$ matrices of the corresponding 2EP. For such problems we can apply the MFRD and solve the 2EP \eqref{eq:2ep_elias} by a subspace method that computes 
a small number of eigenvalues close to the target. We demostrate this approach on \eqref{eq:penta} and $n=100$ (this leads to $\Delta$ matrices of size $10000\times 10000$ is we apply the MFRD, or $20000\times 20000$ if we apply Algorithm 1). The computed 2D-eigenvalues are presented in Figure~\ref{fig:ExampleBig}. The left hand side figure shows the 
 first 100 2D-eigenvalues obtained by the Krylov-Schur method for the 2EP \cite{MP_SylvArnoldi} that we applied to search for the 2D-eigenvalues $(\lambda,\mu)$
 with a target $\lambda_0=-0.15$. As it can be seen from the figure, we indeed
 get 2D-eigenvalues such that $\lambda$ is close to $\lambda_0$. The right hand side figure shows the first eight 2D-eigenvalues obtained by the Jacobi-Davidson
 method with harmonic Ritz values for the 2EP  \cite{HP_HarmJDMEP} using the target
 $(\lambda_0,\mu_0)=(0,5)$. For more details and the choice of the parameters, see the Matlab code in the repository with
 the numerical examples. 

The computed four 2D-eigenvalues from the right hand side of Figure \ref{fig:ExampleBig} that are closest to $(0,5)$ are:
 \begin{align*}
 (\lambda_1,\mu_1)&=(          -5.3655644583\cdot 10^{-2}, 4.99203880712),\\
 (\lambda_2,\mu_2)&=(\phantom{-}5.5000908314\cdot 10^{-2}, 5.00489644932),\\
 (\lambda_3,\mu_3)&=(          -3.3513752079\cdot 10^{-2}, 5.04903430286),\\
 (\lambda_4,\mu_4)&=(\phantom{-}3.2092945395\cdot 10^{-2}, 4.94904259728).
 \end{align*}
 
\begin{figure}[h]
    \centering
    \includegraphics[width=6.4cm]{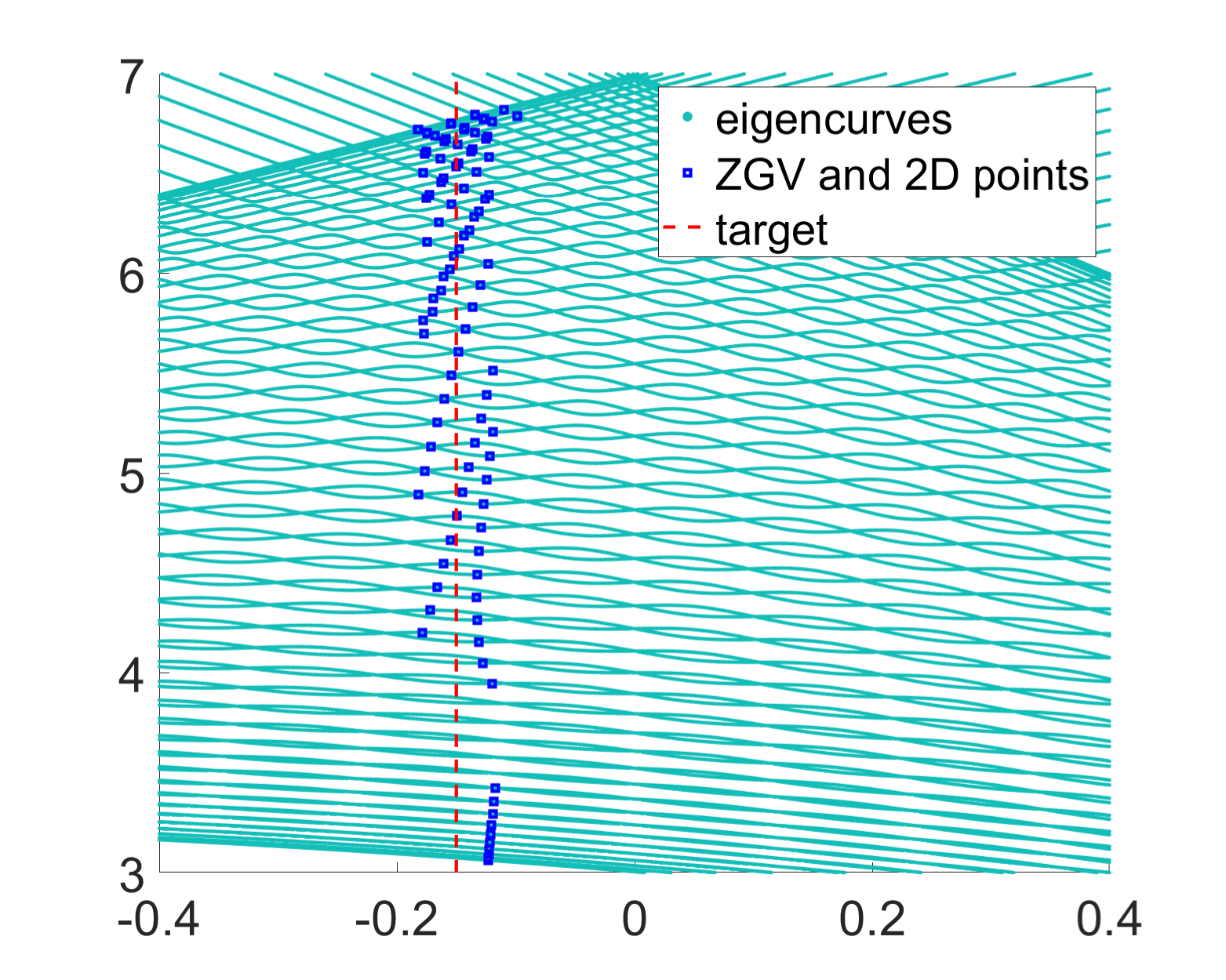}\
    \includegraphics[width=6.4cm]{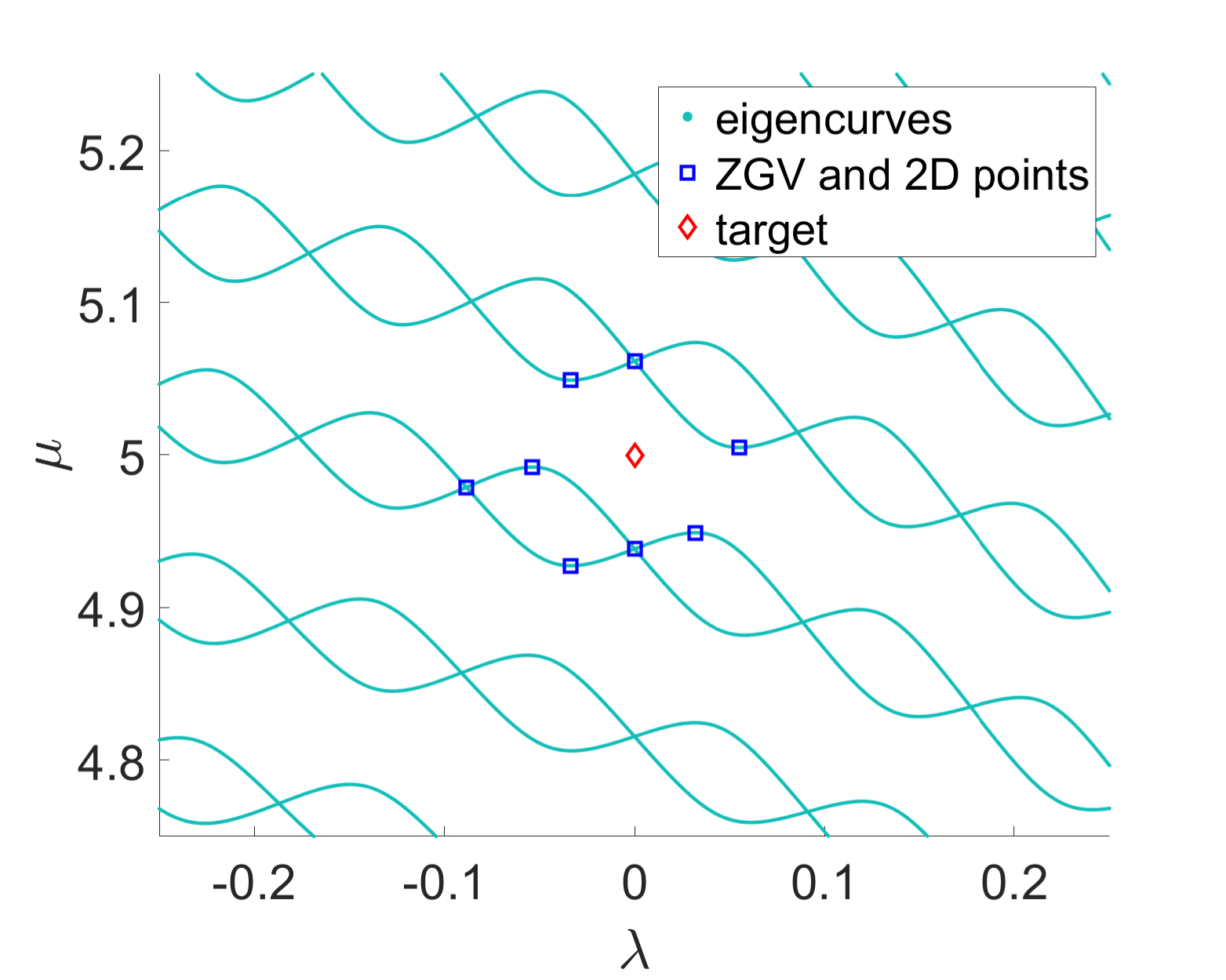}
    \caption{2D-eigenvalues of \eqref{eq:penta} for $n=100$. On the left hand side are first 100 solutions obtained with
    the Krylov-Schur method and the target $\lambda_0=-0.15$, and on the right hand side are
    first 8 solutions obtained with
    the Jacobi-Davidson method and the target $(\lambda_0,\mu_0)=(0,5)$}
    \label{fig:ExampleBig}
\end{figure}
\end{example}

To access the influence of the parameter $\delta$ in the MFRD, we applied the method to 
\eqref{eq:penta} with $n=10$ using values $\delta=10^{-k}$ for $k=1,\ldots,10$. The problem has 39 ZGV points and  25 additional 2D points of type d). For each $\delta$ we performed 10 runs (the results vary due to random processes in the MEP solver from MultiParEig and the 
random vectors $a$ and $b$ in the Gauss-Newton method)
and computed the empirical success rate, defined as the average ratio of 
recovered points (for all 2D points, 2D points of type d), and ZGV points,
respectively). The results are summarized in Table~\ref{tab:toeplitz}.  

We observe that $\delta$ should not be too small nor too large. If $\delta$ is too small, the MEP is nearly singular and the computed eigenvalues are inaccurate. If $\delta$ is too large
the eigenvalues may be perturbed too far away for the Gauss-Newton method to converge.
The success rate of recovering 2D points of type d) is higher for smaller $\delta$ because
the corresponding eigenvalues have higher multiplicities and thus perturb more readily.
On the other hand, the eigenvalues corresponding to ZGV points are simple and easier to detect when $\delta$ is larger.

\begin{table}[htb!] 
\centering
\caption{Empirical success rate of recovering 2D and ZGV points of the pencil $A-\lambda B -\mu I$ for matrices 
\eqref{eq:penta} with $n=10$ using the MFRD and Gauss-Newton.\label{tab:toeplitz}}
\begin{tabular}{cccc|cccc} \hline \rule{0pt}{2.85ex}%
$\delta$ & $P_{\rm all}(\delta)$ & $P_{\rm 2D(d)}(\delta)$ & $P_{\rm ZGV}(\delta)$ & 
$\delta$ & $P_{\rm all}(\delta)$ & $P_{\rm 2D(d)}(\delta)$ & $P_{\rm ZGV}(\delta)$ \\[0.5mm]
\hline \rule{0pt}{2.5ex}%
$10^{-1}$ & $0.5852$ & $0.2320$ & $0.8115$ & $10^{-6}$  & $1.0000$ & $1.0000$ & $1.0000$ \cr 
$10^{-2}$ & $0.8070$ & $0.5060$ & $1.0000$ & $10^{-7}$  & $0.9891$ & $1.0000$ & $0.9821$ \cr
$10^{-3}$ & $0.9273$ & $0.8140$ & $1.0000$ & $10^{-8}$  & $0.7797$ & $1.0000$ & $0.6385$ \cr
$10^{-4}$ & $1.0000$ & $1.0000$ & $1.0000$ & $10^{-9}$  & $0.5281$ & $1.0000$ & $0.2256$ \cr
$10^{-5}$ & $1.0000$ & $1.0000$ & $1.0000$ & $10^{-10}$ & $0.5398$ & $0.9800$ & $0.2577$ \\
\hline
\end{tabular}
\end{table}

\subsection{Distance to instability}
The 2DEVP is related to the  problem of computing the distance to instability of a stable matrix. If $A\in\CC^{n\times n}$ is a stable matrix, i.e., $\re({\lambda})<0$ for all
eigenvalues $\lambda$ of $A$, then
\[\beta(A)=\min\left\{\| E\|_2:\  A+E \textrm{ is unstable}\right\}.
\]
It is well known, see, e.g., \cite{VanLoan_StableMatrix}, that 
$\beta( A)=\min_{\lambda\in\RR}(\sigma_{\rm min}( A-\lambda i I)),$
which is equivalent to
\begin{equation}\label{eq:mu_minimum}
\beta(A)=\min_{\lambda\in\RR}(\lambda_{\rm sp}(\widetilde A-\lambda \widetilde B)),
\end{equation}
where $\lambda_{\rm sp}$ denotes the smallest positive eigenvalue,
\begin{equation}\label{ex:abtilda}
\widetilde A=\left[\begin{matrix}0 & A\cr A^H & 0\end{matrix}\right]\quad
\textrm{and}\quad
\widetilde B=\left[\begin{matrix}0 & iI\cr -iI & 0\end{matrix}\right].\end{equation}
Many numerical methods were proposed for the computation of $\beta(A)$, 
for an overview see, e.g., \cite{Freitag_Spence_LAA_2011}.
Su, Lu and Bai showed in 
\cite[Thm.~3.1]{LuSuBai_SIMAX} that if $\lambda_0$ minimizes \eqref{eq:mu_minimum}, then 
$(\lambda_0,\mu_0)$, where $\mu_0=\lambda_{\rm sp}(\widetilde A-\lambda_0\widetilde B)$, is a 2D-eigenvalue 
of the 2DEVP for $2n\times 2n$ matrices $\widetilde A,\widetilde B$. Based on that, 
$\beta(A)$ can be computed with a generalization of the Rayleigh quotient iteration
for the computation of 2D-eigenvalues from \cite{LuSuBai_SIMAX}.
As each 2D-eigenvalue is a 2D point, this means that for $\beta(A)$ we have to find the real 2D point 
$(\lambda_0,\mu_0)$ of
$\widetilde A-\lambda \widetilde B -\mu I$ with the smallest $|\mu_0|$.

\begin{example}\label{ex:spence_freitag5}\rm
We consider the matrix $$
A=\left[\begin{matrix} -0.4+6i & 1 & & \cr 1 & -0.1+i & 1 & \cr 0 & 1 & -1-3i & 1 \cr & & 1 & -5+i\end{matrix}\right]$$
from Example 5 in \cite{Freitag_Spence_LAA_2011}. The real eigencurves 
of the pencil $\widetilde A- \lambda \widetilde B -\mu I$ \eqref{ex:abtilda} are presented in Figure
\ref{fig:FS} together with the real 2D points computed with Algorithm 1. As we see on two closeups on the 
right hand side, on the contrary to Example \ref{ex:big} the eigencurves do not intersect, they just come very close and veer apart. This is a well-known behaviour expected in a general case, for more details, see, e.g., \cite{Uhlig_SIMAX}.

\begin{figure}[h]
    \centering
    \includegraphics[width=4cm]{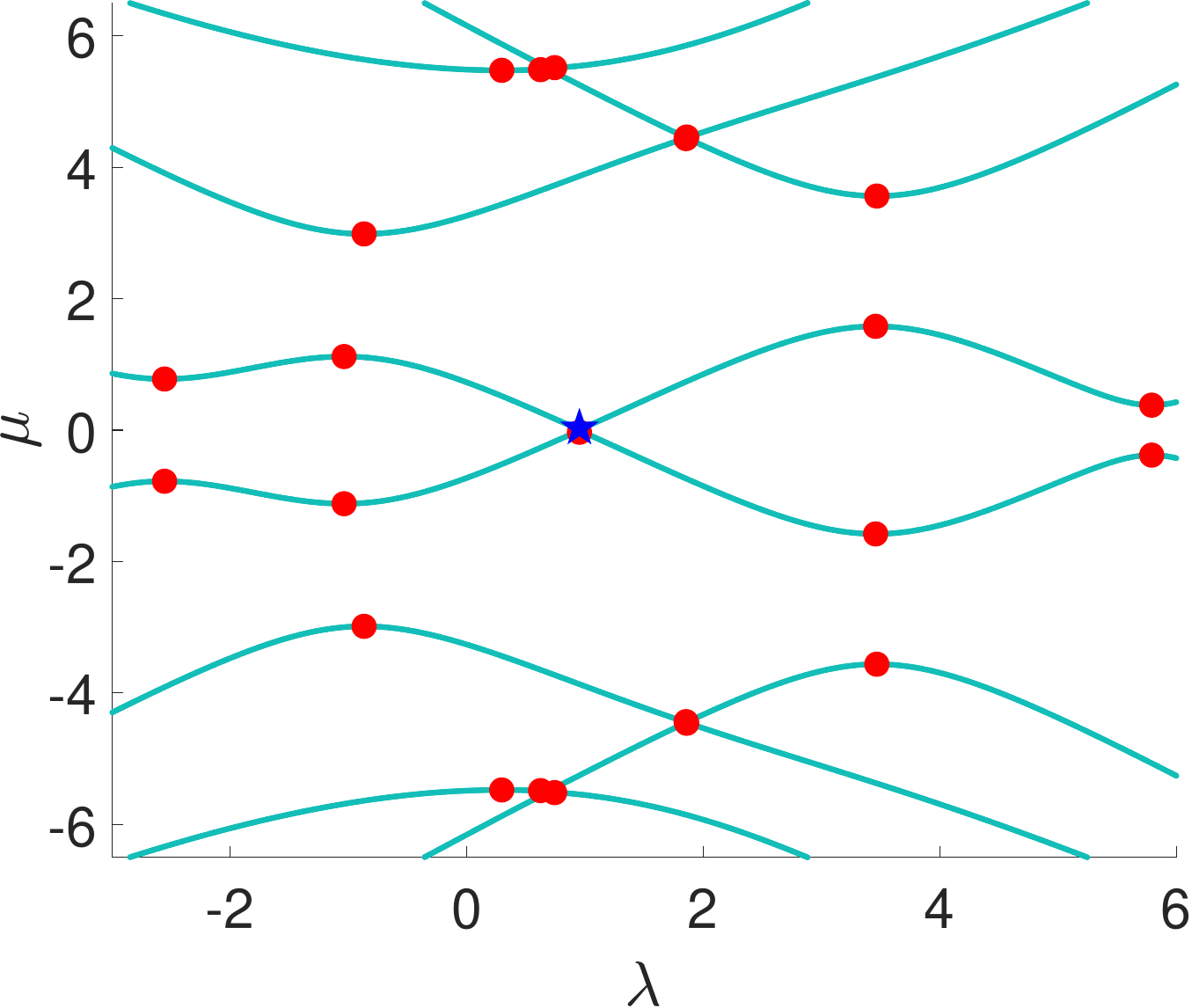}\ \
    \includegraphics[width=4cm]{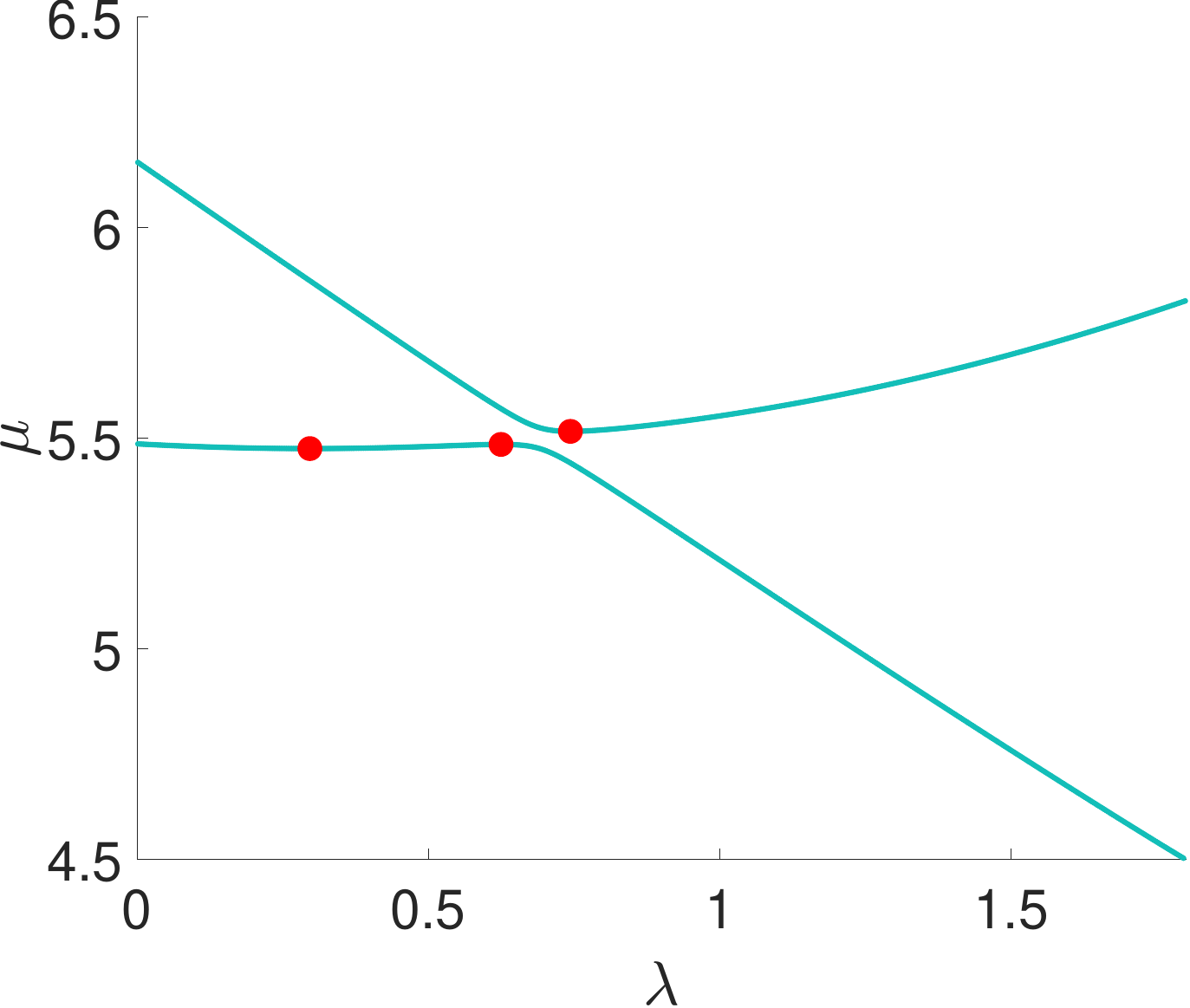}\ \
    \includegraphics[width=4cm]{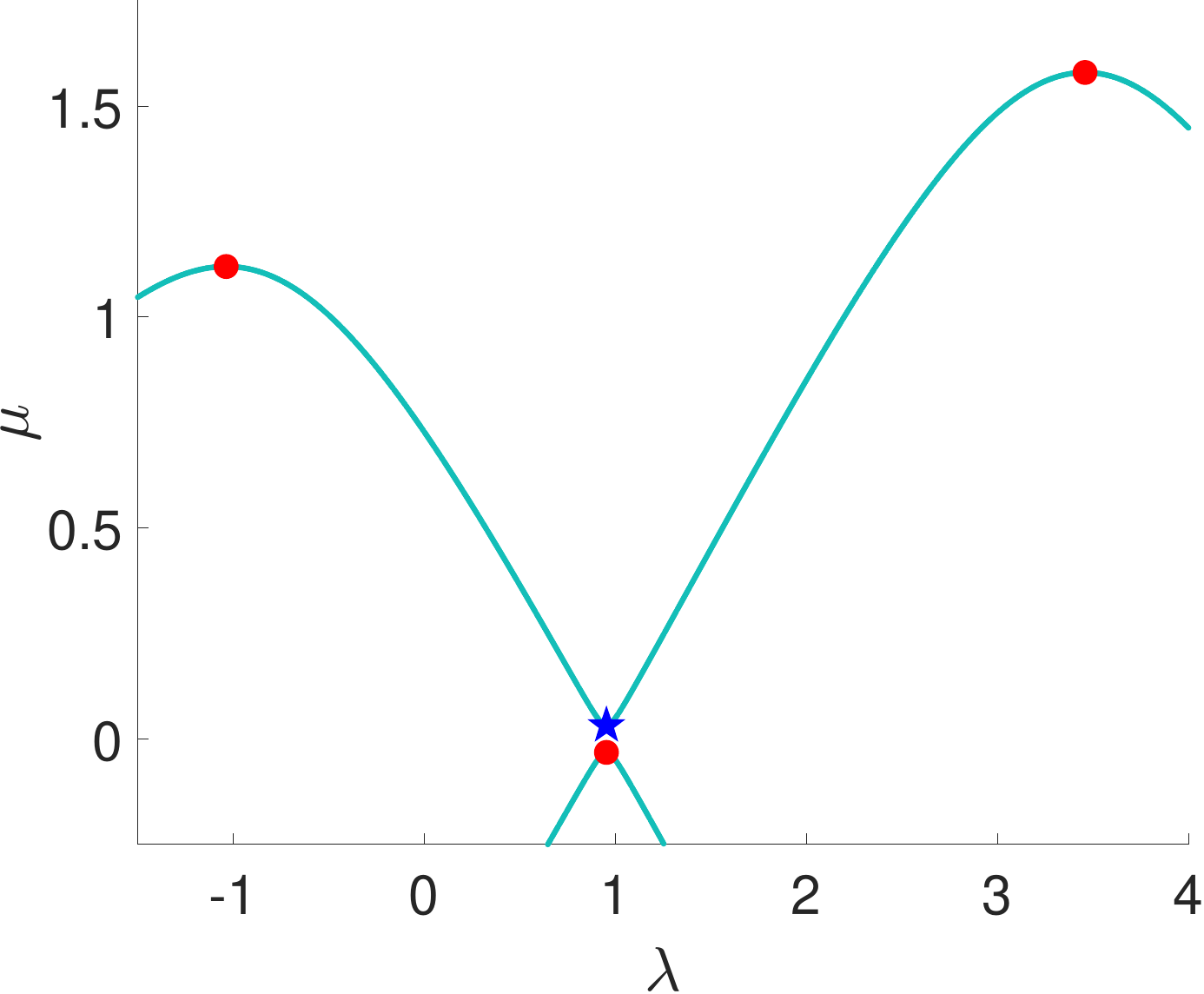}
    \caption{2D-eigenvalues of Example \ref{ex:lusu}}
    \label{fig:FS}
\end{figure}

The method computes all 2D points, which appear in pairs $(\lambda,\pm\mu)$. 
From the 2D point $(0.95301472,0.03188701)$ with the smallest positive $\mu$ (the blue star on Figure~\ref{fig:FS}) we get by further 
refinement by Algorithm 4 that
$\beta(A)=3.188701430320041\cdot 10^{-2}$, which agrees to the result in \cite{Freitag_Spence_LAA_2011}.

Let us remark that this does not lead to an efficient algorithm for the computation of $\beta(A)$. In theory
we could apply Algorithm~3 and use Sylvester-Arnoldi or Krylov-Schur method from \cite{MP_SylvArnoldi} in line 2 to search for the real 2D point $(\lambda_0,\mu_0)$ with the smallest $|\mu_0|$, but in practice this is not an easy task. As it can be seen also from Figure~\ref{ex:spence_freitag5}, we need to compute an interior eigenvalue and, in addition, there 
can be many complex 2D points with a very small $|\mu_0|$.
\end{example}

\subsection{Double eigenvalue problem}

In the double eigenvalue problem, see, e.g.,  \cite{Elias_DoubleEig, MP_DoubleEig}, we have matrices $A,B\in\CC^{n\times n}$ and we are looking for values $\mu$ such 
that $A+\mu B$ has a multiple eigenvalue (generically, such eigenvalues have multiplicity two). 
It follows
that points $(\lambda_0,\mu_0)$, where $\lambda_0$ is a multiple eigenvalue of $A+\mu_0 B$, are exactly 2D points
of the pencil $A+\lambda I +\mu B$. 

Based on the above,  we can compute values $\mu_0$ such that $A+\mu B$ has a multiple eigenvalue from the following 
singular 2EP
\begin{equation}\label{eq:dbleig}
\begin{split}
    (A+\lambda I +\mu B)x & = 0\\
\left(\left[\begin{matrix} A & 0\cr I & A\end{matrix}\right]
+\lambda \left[\begin{matrix} I & 0\cr 0 & I\end{matrix}\right]
+\mu \left[\begin{matrix} B & 0\cr 0 & B\end{matrix}\right]\right)
\left[\begin{matrix} z_1 \cr z_2\end{matrix}\right]&=0.
\end{split}
\end{equation}

The above 2EP for the double eigenvalue problem can be derived like in \cite[Example 7.4]{HMP_SingGep2} from
the property that if $-\lambda_0$ is a multiple eigenvalue of $A+\mu_0 B$ that has algebraic multiplicity one, then there exist nonzero vectors $x$ (an eigenvector) and $y$ (a generalized eigenvector of degree two) such that 
$(A+\mu_0B+\lambda_0I)x=0$ and $(A+\mu_0B+\lambda_0I)y+x=0$.

 A similar approach to solve the double eigenvalue problem
 was used in \cite{MP_DoubleEig}, where a singular 2EP 
 with equations $(A+\lambda I+\mu B)x=0$ and
 $(A+\lambda I +\mu B)^2z=0$ was applied. The second equation is linearized into $(F + \lambda  G + \mu H)z=0$,
 where matrices $F,G,H$ are of size $3n\times 3n$, therefore, the 
 corresponding singular GEPs with $\Delta$-matrices in \cite{MP_DoubleEig} are of size $3n^2\times 3n^2$.
 The approach from \cite{HMP_SingGep2}, which applies \eqref{eq:dbleig}, leads to
 singular GEPs with $\Delta$-matrices of size $2n^2\times 2n^2$ and is more efficient due to smaller matrices. 
 Here we give an even more efficient approach. The key is to apply Algorithm~4, where we have
 to solve nonsingular GEPs with $\Delta$-matrices of size $(2n^2-n)\times (2n^2-n)$. While this is 
 still less efficient than the MFRD approach \cite{Elias_DoubleEig}, where we have to solve 
 a nonsingular 2EPs with corresponding $\Delta$-matrices of size $n^2\times n^2$, 
 it theoretically leads to exact results in exact computation.
 
 \begin{example}\rm\label{ex:doubleeig}
 In Matlab we use {\tt rng(3), A=randn(25), B=randn(25)} to construct two random matrices. There are
 600 values $\mu$ such that $A+\mu B$ has a double eigenvalue and these are the $\mu$ coordinates
 of the 2D points $(\lambda,\mu)$ of the pencil $A+\lambda I+\mu B$. We can compute all the points
 using any of the proposed algorithms, the fastest method is the MFRD combined with the Gauss-Newton method, which
 calculates all points in $1.52$ seconds. 
 In Figure \ref{fig:double} we present all such values $\mu$ inside $[-0.5,0.5]\times [-0.5,0.5]$ that we
 computed by Algorithm 3. For the second computation we used a Krylov-Schur for 2EP in line 2 of Algorithm 3
 to compute six
 solutions closest to the origin that are encircled by red circles. This computation takes $0.19$ seconds and returns the following six 
 solutions $\mu$ with the minimal absolute value:
 \begin{align*}
 \mu_{1,2}&=       -5.1825645307\cdot 10^{-2}\pm 2.1720742563 \cdot 10^{-2}i,\\
 \mu_3 &=\phantom{-}6.9303957823\cdot 10^{-2},\quad \mu_4=              -7.5529261749\cdot 10^{-2},\\
 \mu_{5,6}&=\phantom{-}1.1251034552\cdot 10^{-2}\pm 1.2727979158\cdot 10^{-1}i.
 \end{align*}
This shows that in Algorithm 3 we can apply a subspace method and compute just a small number of solutions.
We can do this also if we apply the MFRD, as we have shown in Example \ref{ex:big}, while in Algorithm 1 this is not possible because we have a singular eigenvalue problem.

 \begin{figure}[h]
    \centering
    \includegraphics[width=6cm]{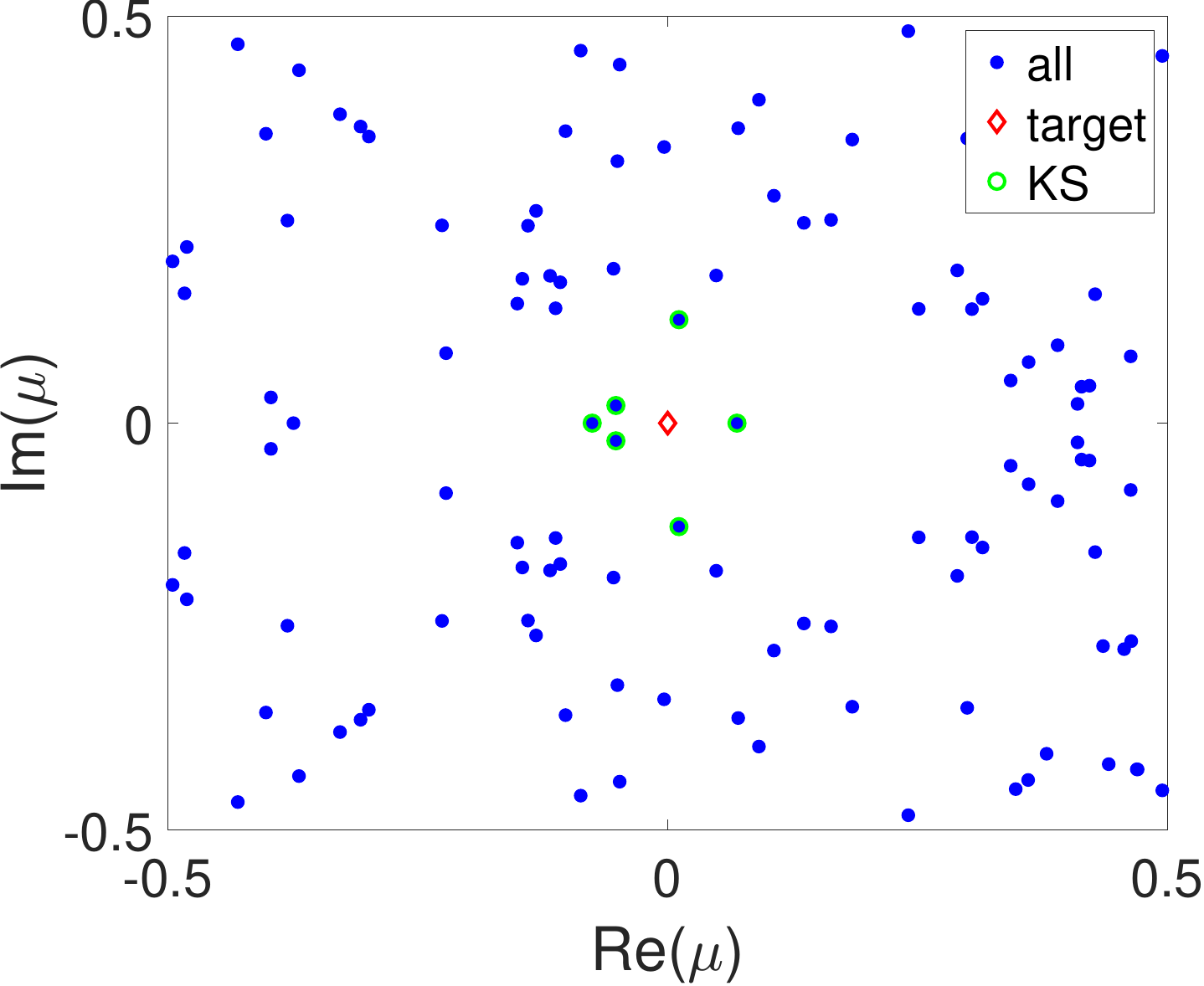}
    \caption{Points $\mu$ in the complex plane such that $A+\mu B$ from Example \ref{ex:doubleeig} has a double eigenvalue}
    \label{fig:double}
\end{figure}

 \end{example}

\subsection{Critical points of two-parameter Sturm-Liouville eigencurves}

We consider a two-parameter Sturm-Liouville problem \cite{AtkinsonMingarelli, BindingVolkmer1996} of the form
\begin{equation}\label{eq:SL_de}
-(p(x)y')'+q(x)y=(\lambda r(x) + \mu) y,\quad a\le x\le b,
\end{equation}
with boundary conditions 
\begin{equation}\label{eq:SL_bc}
\cos(\alpha)y(a)-\sin(\alpha)p(a)y'(a)=0,\quad
\cos(\beta)y(b)-\sin(\beta)p(b)y'(b)=0,
\end{equation}
where $p$ is continuously differentiable and positive on $[a,b]$, $q$ and $r$ are piecewise continuous on $[a,b]$, and $\alpha,\beta$ are real.
The set of $(\lambda,\mu)\in\RR^2$, for 
which there exist a nontrivial solution $y$ that satisfies \eqref{eq:SL_de} and \eqref{eq:SL_bc}, forms eigencurves that are a countable union of graphs of analytic functions, for details see, e.g., \cite{BindingVolkmer1996}. 

For a fixed $\lambda\in\RR$ we get a regular one-parameter Sturm-Liouville problem \eqref{eq:SL_de}, \eqref{eq:SL_bc}. It is well known that for each $n\in{\mathbb N}$ there
exists exactly one eigenvalue $\mu_n(\lambda)$ such that the corresponding eigenfunction $y$ has exactly 
$n-1$ zeros in $(a,b)$.
We also know that
$\mu_j(\lambda)<\mu_{j+1}(\lambda)$ for each $j\in{\mathbb N}$.
It turns out \cite[Thm.~2.1]{BindingVolkmer1996} that eigencurves $\mu_n(\lambda)$ for $n\in{\mathbb N}$ 
 are analytic functions. We are interested in critical points, where
$\mu_n'(\lambda)=0$, which are clearly generalizations of ZGV points. Thus, we can 
discretize \eqref{eq:SL_de}, \eqref{eq:SL_bc} into 
a two-parameter pencil $A+\lambda B+\mu I$ and then compute the corresponding ZGV points.

\begin{example}\rm\label{ex:Mathieu}
We consider the modified Mathieu equation
\begin{equation}\label{eq:Mathieu}
    y''(x)-2\lambda \cos(2x)y(x) +\mu y(x) = 0
\end{equation}
with boundary conditions $y'(0)=y'(\pi/2)=0$. The first five dispersion curves $\mu_j(\lambda)$ for $j=1,\ldots,5$ together with real ZGV points are presented in Figure \ref{fig:Mathieu}. Beside the 
trivial ZGV points of the form $(0,4k^2)$ for $k=0,\ldots,4$ the remaining ZGV points on the figure are 
$(\pm 11.14606106, 17.41358458)$, 
$(\pm 31.48781869, 42.39762508)$, and
$(\pm 60.12377598, 78.78937721)$. 

To obtain the above ZGV points, we discretized \eqref{eq:Mathieu} by the spectral collocation on $n=25$ points using {\tt bde2mep} function in MultiParEig \cite{multipareig_28}, see \cite{GHPR_SpecColl_AMC} for details, and applied Algorithm 1. We refined the solutions using finer discretizations on $n=50$ and $n=100$ points, where each time a ZGV point on a coarser grid was used as an initial approximation for Algorithm 4. 

\begin{figure}[h]
    \centering
    \includegraphics[width=6cm]{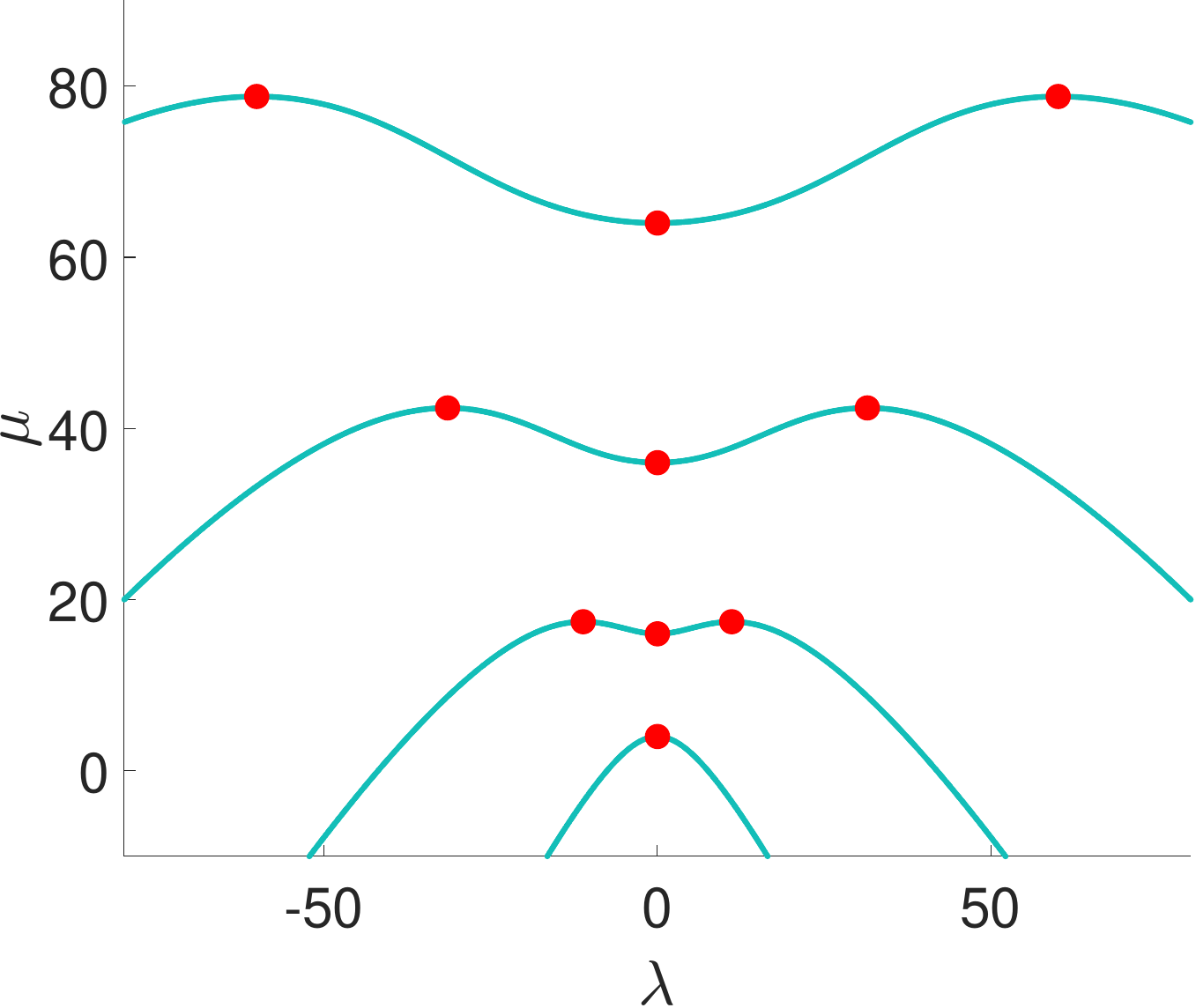}
    \caption{ZGV points for the boundary value problem \eqref{eq:Mathieu} }
    \label{fig:Mathieu}
\end{figure}
\end{example}

\subsection{ZGV points for quadratic eigenvalue problems}
The expression ZGV originates from engineering applications.  
In the study of anisotropic elastic waveguides, see, e.g., \cite{prada_local_2008}, we 
consider an eigenvalue problem
\begin{equation}\label{eq:ZGV1}
\big(\lambda^2 L_2+ \lambda L_1+L_0+\omega^2 M\big)\,u=0,
\end{equation}
where $L_0$, $L_1$, $L_2$, $M$ are Hermitian $n\times n$ matrices such that  
  $L_2$ and $M$ are nonsingular, obtained by a discretization
of a boundary value problem. The solution
 are dispersion curves $\omega=\omega(\lambda)$. 
 Of particular interest are the {\em zero-group-velocity} (ZGV) points, where 
$\omega$ and $\lambda$ are real, and 
$\omega'(\lambda)=0$. If we assume that $u=u(\lambda)$ and $\omega(\lambda)$ are differentiable, we obtain by differentiating 
\eqref{eq:ZGV1} that at a ZGV point $(\lambda,\omega)$ it holds
\begin{equation}\label{eq:ZGV2}
\big(\lambda^2\, \widetilde L_2+ \lambda\, \widetilde L_1+\widetilde L_0+\omega^2\, \widetilde M\big)\,\widetilde u=0,
\end{equation}
where
\[
\widetilde{L}_2=\left[\begin{matrix}L_2 & 0 \\ 0 & L_2\end{matrix}\right],\quad
\widetilde{L}_1=\left[\begin{matrix}L_1 & 0 \\ 2L_2 & L_1\end{matrix}\right],\quad
\widetilde{L}_0=\left[\begin{matrix}L_0 & 0 \\ L_1 & L_0\end{matrix}\right],\quad
\widetilde{M}=\left[\begin{matrix}M & 0 \\ 0 & M\end{matrix}\right],\quad
\widetilde{u}=\left[\begin{matrix}u \\ u'\end{matrix}\right].
\]
Equations \eqref{eq:ZGV1} and \eqref{eq:ZGV2} form 
a quadratic 2EP \cite{MP_Q2EP, HMP_LinQ2MEP} and
 numerical methods similar to the presented in this paper can be derived for the computation of ZGV points of \eqref{eq:ZGV1}. For more details, see \cite{ZGV_JASA_23}, where a method for the above quadratic 2EP is presented together with a generalization of the MRFD from Section \ref{sec:elias} 
and a generalized Gauss--Newton method from Section \ref{sec:Gauss_Newton}.

Although the next approach is less efficient than the one in \cite{ZGV_JASA_23}, which exploits
the connection to 
a singular quadratic 2EP,
we can substitute $\mu=\omega^2$ and apply 
any of the standard linearizations for the quadratic eigenvalue problem 
to transform (\ref{eq:ZGV1}) into a problem of form (\ref{eq:ABC}) with $2n\times 2n$ matrices. We see from
$\mu'(\lambda)=2\omega(\lambda)\omega'(\lambda)$ that  each ZGV
point of (\ref{eq:ZGV1}) corresponds to a ZGV point of the linearized pencil, which can 
in addition have ZGV points with $\mu=0$.
For instance,
we can write (\ref{eq:ZGV1}) as
\begin{equation}\label{eq:linQEP}
    \left(\left[\begin{matrix}L_0 & L_1\cr 0 & -I\end{matrix}\right]
    +\lambda \left[\begin{matrix}0 & -L_2\cr I & 0\end{matrix}\right] 
    +\mu \left[\begin{matrix}M & 0\cr 0 & 0\end{matrix}\right]\right)
    \left[\begin{matrix}u  \cr \lambda u\end{matrix}\right] =0
\end{equation}
and apply Algorithm 1 or Algorithm 4 to (\ref{eq:linQEP}). We illustrate 
this with a small example.

\begin{example}\rm\label{ex:QEP}
The problem \eqref{eq:ZGV1} with the following diagonal and tridiagonal matrices
\[{
L_2=\left[\begin{matrix}-1 & 0.5 & \cr 0.5 & -2 & 0.5 \cr & 0.5 & -3\end{matrix}\right],\
L_1=\left[\begin{matrix}1 & -0.25 & \cr -0.25 & 2 & -0.25 \cr & -0.25 & -3\end{matrix}\right],\
L_0=\left[\begin{matrix}-1 &  & \cr  & -2 &  \cr & & -3\end{matrix}\right],\
M=\left[\begin{matrix}2 & 1 & \cr 1 & 3 & 1 \cr & 1 & 4\end{matrix}\right]}
\]
has five real ZGV points $(\lambda,\omega)$ such that $\omega>0$, which are presented
in Figure \ref{fig:QEP}. By
applying Algorithm 1 to the linearization \eqref{eq:linQEP} we compute these
ZGV points as:
 \begin{align*}
 (\lambda_1,\omega_1)&=(          -0.2312197373, 0.79089022421),\
 (\lambda_2,\omega_2)=(0.3684223373, 0.82195756940),\\
 (\lambda_3,\omega_3)&=(\phantom{-}0.6315720581, 0.54233673936),\
 (\lambda_4,\omega_4)=(0.1584790129, 0.82797266404),\\
 (\lambda_5,\omega_5)&=(\phantom{-}0.1200999663, 1.10785496051).
 \end{align*}
 
 \begin{figure}[h]
    \centering
    \includegraphics[width=6cm]{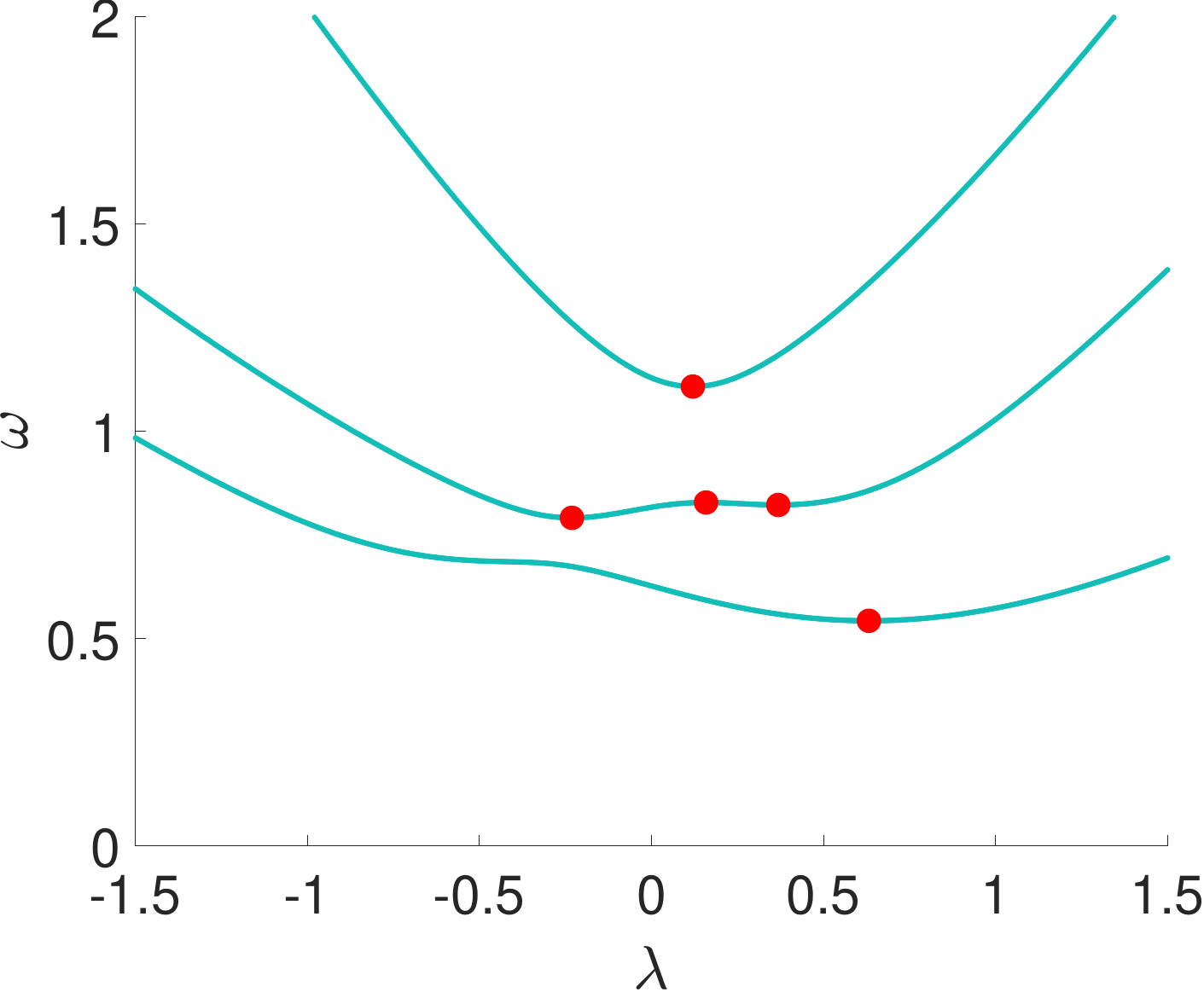}
    \caption{Real ZGV points $(\lambda,\omega)$ with positive $\omega$ for the parameter-dependent quadratic eigenvalue 
    problem from Example \ref{ex:QEP}}
    \label{fig:QEP}
\end{figure}
\end{example}

\section{Conclusions}\label{sec:conclusion}
We have investigated critical points of eigencurves of $n\times n$ bivariate matrix pencils and provided
common theory that links ZGV points, 2D-eigenvalues from \cite{LuSuBai_SIMAX}, and the newly introduced 2D points. We derived a singular 2EP, which is challenging to solve, whose solutions represent the critical points.

We proposed three numerical methods for the
computation of critical points. 
The first method in Algorithm 1 computes all critical points from eigenvalues of a singular GEP of size 
$2n^2\times 2n^2$ that is built from operator determinants related to the singular 2EP. Eigenvalues of this singular GEP 
are computed via a random projection to a GEP of the size of the normal rank from \cite{HMP_SingGep2}.
Because of the high complexity it is in practice feasible only for small problems.
The second method in Algorithm 3 follows the same idea, but due to a structured random projection
requires the above singular GEP only implicitly.
Furthermore, by projecting it into a regular 2EP, this approach is more efficient than Algorithm 1 for the computation of all
critical points, and, unlike Algorithm 1, it can be used to compute just a small number of solutions close to a target. Both of the above algorithms can find all critical points in exact computation.

The third approach is a locally convergent Gauss-Newton-type method, which exhibits quadratic convergence near ZGV points. Combining this method with the method of fixed relative distance from \cite{Elias_DoubleEig} that provides initial approximations yields 
a solver faster than the above, which is suitable for larger problems and is likely to find all critical points. 
The Gauss-Newton-type method can as well be applied to refine the solutions obtained by Algorithm 1 and Algorithm 3.

We presented many possible applications. Through extensive numerical experiments, we have demonstrated that the proposed 
numerical methods can successfully 
compute either all critical points or a subset of critical points close to a given target.  
\medskip

\noindent\textbf{Acknowledgments} The author would like to thank the anonymous referees for their helpful comments and suggestions.
\smallskip

\noindent\textbf{Funding}\quad {Funding was provided by Javna agencija za znanstvenoraziskovalno in inovacijsko dejavnost Republike Slovenije (Grant Nos. N1-0154 and P1-0294).}
\smallskip

\noindent\textbf{Code availability}\quad The code and data for numerical examples are available 
at \url{https://github.com/borplestenjak/ZGV_Points}.
\smallskip

\section*{Declarations}

\noindent\textbf{Conflict of interests}\quad The author has no conflict of interest to declare that are relevant to the content of this article.



\appendix
\section{Generic situation}

The following theorem relies on results from algebraic geometry that can be found in, e.g., \cite{Hartshorne_AG} and \cite{Shafarevich1}.

\begin{theorem}\label{thm:appendix} 
There exist generic sets $\Omega_3\subset\Omega_2\subset\Omega_1\subset (\CC^{n\times n})^3$ such that:
\begin{enumerate}
\item[1)] For all $(A,B,C)\in\Omega_1$, the matrices $B$ and $C$ are nonsingular, and
the GEP $(A+\lambda_0B)x + \mu Cx=0$ has a multiple eigenvalue in $\mu$ for only finitely many $\lambda_0\in\CC$.
\item[2)] For all $(A,B,C)\in\Omega_2$,
the bivariate pencil 
$A+\lambda B+\mu C$ has exactly $n(n-1)$ distinct 2D points. 
\item[3)] For all $(A,B,C)\in\Omega_3$, 
if $(\lambda_0,\mu_0)\in\CC^2$ is a 2D point of 
$A+\lambda B+\mu C$, then $(\lambda_0,\mu_0)$ is a ZGV point
and $\lambda_0$ is a double eigenvalue in $\lambda$ of the GEP $(A+\mu_0 C)x+\lambda Bx=0$.
\end{enumerate}
\end{theorem}

\begin{proof}
Let $f(\lambda,\mu,A,B,C):=\det(A+\lambda B+\mu C)$.
\begin{enumerate}
\item[1)] We consider $f$ as a polynomial in $\mu$ whose coefficients are polynomials in $\lambda$ and the entries of $A,B,C$. Then $f$ has a multiple root if the resultant $\mathrm{Res}_\mu(f,f_\mu)$ of $f$ and its derivative $f_\mu=\frac{\partial f}{\partial \mu}$ 
is zero. Since $\mathrm{Res}_\mu(f,f_\mu)$ is a polynomial in $\lambda$ whose coefficients are polynomials in the entries of $A,B,C$, 
this happens for only finitely many $\lambda$ unless 
$\mathrm{Res}_\mu(f,f_\mu)\equiv 0$. We can take 
$\Omega_1=(\CC^{n\times n})^3\backslash (S_{1,a}\cup S_{1,b}\cup S_{1,c})$,
where
$S_{1,a}=\{(A,B,C):\ \det(B)=0\}$,
$S_{1,b}=\{(A,B,C):\ \det(C)=0\}$, and
$S_{1,c}=\{(A,B,C):\ \mathrm{Res}_\mu(f,f_\mu)\equiv 0\}$
are algebraic sets and thus $\Omega_1$ is generic.

\item[2)] Let $(A,B,C)\in\Omega_1$. Then 
$B$ and $C$ are nonsingular and,  
by Lemma \ref{lem:dbl}, $(\lambda_0,\mu_0)\in\CC^2$ is a 2D point if and only if $\lambda_0$ is a multiple eigenvalue of the GEP $(A+\mu_0 C)x+\lambda Bx=0$. If we now consider $f$ as a polynomial in $\lambda$, then
$g(\mu,A,B,C):=\mathrm{Res}_\lambda(f,f_\lambda)$, where 
$f_\lambda=\frac{\partial f}{\partial \lambda}$, is a polynomial in $\mu$ of degree $n(n-1)$ unless the leading coefficient, which
is a polynomial in the entries of $A,B,C$, vanishes. We introduce
algebraic sets
$S_{2,a}=\{(A,B,C):\ \textrm{coefficient of}\ \mu^{n(n-1)}\ \textrm{in}\ g\ \textrm{is zero}\}$
and 
$S_{2,b}=\{(A,B,C):\ \mathrm{Res}_\mu(g,g_\mu)=0\}$,
and take $\Omega_2=\Omega_1\backslash (S_{2,a}\cup S_{2,b})$.

Then
$\Omega_2$ is generic and for each $(A,B,C)\in\Omega_2$ the polynomial
$g$ has $n(n-1)$ distinct roots. To each root $\mu_0$ there corresponds
at least one 2D point $(\lambda_0,\mu_0)$, where
$\lambda_0$ is a multiple eigenvalue of the GEP $(A+\mu_0 C)x+\lambda B x=0$.
Since 2D points are the roots of the system of polynomials
$f(\lambda,\mu,A,B,C)=0$ and $f_\lambda(\lambda,\mu,A,B,C)=0$, which
cannot have more than $n(n-1)$ roots by B\'ezout's theorem, 
it follows that
$A+\lambda B +\mu C$ has exactly $n(n-1)$ distinct 2D points.

\item[3)]
We observe that $f$ as a polynomial in $\lambda$
has a root of multiplicity three or higher  if
in addition to $\mathrm{Res}_\lambda(f,f_\lambda)$ also its first subresultant, which is again a
polynomial in $\mu$ and the entries of $A,B,C$, is zero, i.e.,
$\mathrm{Res}_\lambda(f,f_\lambda)=\mathrm{Sres}_{1,\lambda}(f,f_\lambda)=0$.
This happens only if $\mathrm{Res}_\mu\big(\mathrm{Res}_\lambda(f,f_\lambda),\mathrm{Sres}_{1,\lambda}(f,f_\lambda)\big)=0$, which is a polynomial condition in the entries of $A,B,C$ that we use to define the algebraic set $S_{3,a}$.

To exclude the possibility of 2D points $(\lambda_0,\mu_0)$ such that 
$\mu_0$ is a multiple eigenvalue of the GEP $(A+\lambda_0 B)x+\mu Cx=0$, i.e.,
$(\lambda_0,\mu_0)$ is a common root
of $f, f_\lambda$, and $f_\mu$, we introduce the algebraic set
$S_{3,b}$ in $\CC\times \CC\times (\CC^{n\times n})^3$ of 
 all $(\lambda,\mu,A,B,C)$ such that  
$f(\lambda,\mu,A,B,C)=0$, $f_\lambda(\lambda,\mu,A,B,C)=0$, and 
$f_\mu(\lambda,\mu,A,B,C)=0$.
The Zariski closure $\overline{\Pi(S_{3,b})}$ of its projection $\Pi(S_{3,b})$ onto
$(\CC^{n\times n})^3$ is algebraic. To show that $\overline{\Pi(S_{3,b})}$ is proper, we need
$(A_0,B_0,C_0)\in\Omega_2$ such that $f_\mu(\lambda,\mu,A_0,B_0,C_0)\ne 0$ 
for all 2D points of 
$A_0+\lambda B_0+\mu C_0$.

Due to Dixon's result on determinantal representations \cite{Dixon_1902}, there exist $n\times n$ matrices $A_0,B_0,C_0$ such that
\begin{equation}\label{eq:gen_primerABC}
f(\lambda,\mu,A_0,B_0,C_0)=\lambda^n-\lambda-\mu^n+1.
\end{equation}
We see that $f_\lambda(\lambda,\mu)=n\lambda^{n-1}-1$  
vanishes when $\lambda^{n-1}=\frac{1}{n}$, which 
gives $n-1$ simple solutions $\lambda$. By inserting such
$\lambda$ in
 \eqref{eq:gen_primerABC}, we see that corresponding
$\mu$ satisfy the equation $\mu^n=1-\frac{n-1}{n}\lambda$, which
has $n$ nonzero simple solutions. This way we obtain $n(n-1)$ distinct 2D points.
Since $f_\mu(\lambda,\mu)=-n\mu^{n-1}$ and all 2D points have nonzero $\mu$,
it follows that $f_\mu$ is nonzero for all 2D points of $A_0+\lambda B_0+\mu C_0$.

Then, by continuity and implicit function theorem, 
$f_\mu$ is nonzero 
for all 2D points of the bivariate pencil 
$A+\lambda B+\mu C$
for all 
$(A,B,C)$ in a sufficiently small open Euclidean
neighborhood of $(A_0,B_0,C_0)$. This shows that
$(A_0,B_0,C_0)$ does not lie in the Zariski closure of $\Pi(S_{3,b})$, therefore
$\overline{\Pi(S_{3,b})}$ is proper and we can take 
$\Omega_3=\Omega_2\backslash (S_{3,a}\cup \overline{\Pi(S_{3,b})})$.
\end{enumerate}
\end{proof}

Note that matrices $(A,B,C)\in\Omega_1$ satisfy the assumptions of Theorem \ref{thm:nrank_eig}. 
If $(A,B,C)\in\Omega_3$, then all 2D points are ZGV points, i.e., of type a) from Remark \ref{rem:abc_2D_cases}, which is the generic case.

\end{document}